\documentclass[a4paper,11pt,reqno]{amsart}
\usepackage[utf8]{inputenc}
\usepackage{amsmath}
\usepackage{amssymb}
\usepackage{amsthm}
\usepackage{hyperref}
\usepackage{color}
\usepackage{a4wide}
\usepackage[arrow, matrix, curve]{xy}

\makeatletter

\newcommand{\fraka}{\mathfrak{a}}
\newcommand{\frakg}{\mathfrak{g}}

\newcommand{\frakl}{\mathfrak{l}}
\newcommand{\frakm}{\mathfrak{m}}
\newcommand{\frakn}{\mathfrak{n}}
\newcommand{\frako}{\mathfrak{o}}

\newcommand{\CC}{\mathbb{C}}

\newcommand{\HH}{\mathbb{H}}
\newcommand{\NN}{\mathbb{N}}

\newcommand{\RR}{\mathbb{R}}

\newcommand{\calD}{\mathcal{D}}

\newcommand{\calF}{\mathcal{F}}

\newcommand{\calS}{\mathcal{S}}

\newcommand{\calU}{\mathcal{U}}

\newcommand{\0}{{\bf 0}}
\renewcommand{\1}{{\bf 1}}

\DeclareMathOperator{\so}{\mathfrak{so}}
\DeclareMathOperator{\Ind}{Ind}
\DeclareMathOperator{\tr}{tr}

\DeclareMathOperator{\ad}{ad}
\DeclareMathOperator{\Ad}{Ad}

\DeclareMathOperator{\Hom}{Hom}

\DeclareMathOperator{\id}{id}

\DeclareMathOperator{\diag}{diag}

\DeclareMathOperator{\PV}{P.V.}
\renewcommand\Re{\operatorname{Re}}

\newsavebox{\@brx}
\newcommand{\llangle}[1][]{\savebox{\@brx}{\(\m@th{#1\langle}\)}%
  \mathopen{\copy\@brx\kern-0.5\wd\@brx\usebox{\@brx}}}
\newcommand{\rrangle}[1][]{\savebox{\@brx}{\(\m@th{#1\rangle}\)}%
  \mathclose{\copy\@brx\kern-0.5\wd\@brx\usebox{\@brx}}}
  
\makeatother

\theoremstyle{plain}
\newtheorem{theorem}{Theorem}[section]
\newtheorem{proposition}[theorem]{Proposition}
\newtheorem{lemma}[theorem]{Lemma}
\newtheorem{corollary}[theorem]{Corollary}

\newtheorem{fact}[theorem]{Fact}

\newtheorem{thmalph}{Theorem}

\theoremstyle{definition}

\numberwithin{equation}{section}

\title{An extension problem related to the fractional Branson--Gover operators}

\author{Jan Frahm}
\address{Department Mathematik, FAU Erlangen-N\"{u}rnberg, Cauerstr. 11, 91058 Erlangen, Germany}
\email{frahm@math.fau.de}

\author{Bent {\O}rsted}
\address{Institut for Matematiske Fag, Aarhus Universitet, Ny Munkegade 118, 8000 Aarhus C, Denmark}
\email{orsted@math.au.dk}

\author{Genkai Zhang}
\address{Mathematical Sciences, Chalmers University of Technology and the University of Gothenburg, SE-412 96 G\"{o}teborg, Sweden}
\email{genkai@chalmers.se}

\thanks{Research by G. Zhang partially supported by the Swedish Science Council (VR)}

\begin{document}

\begin{abstract}
The Branson--Gover operators are conformally invariant differential operators of even degree acting on differential forms. They can be interpolated by a holomorphic family of conformally invariant integral operators called \emph{fractional Branson--Gover operators}. For Euclidean spaces we show that the fractional Branson--Gover operators can be obtained as Dirichlet-to-Neumann operators of certain conformally invariant boundary value problems, generalizing the work of Caffarelli--Silvestre for the fractional Laplacians to differential forms. The relevant boundary value problems are studied in detail and we find appropriate Sobolev type spaces in which there exist unique solutions and obtain the explicit integral kernels of the solution operators as well as some of its properties.
\end{abstract}

\maketitle

\section*{Introduction}

Classical harmonic analysis in Euclidean space deals to a large extent with the analysis and geometry of the Laplace operator; for boundary value problems for harmonic functions one studies Poisson integral operators, and also analogous problems involving fractional powers of the Laplacian have become very important in recent years. Not only Euclidean geometry plays a role here, but also conformal geometry -- and singular elliptic boundary value problems lead to new insight about exactly the fractional Laplacians as observed by Caffarelli--Silvestre in their influencial paper~\cite{CS07}. Their observations have had a huge impact within the PDE community and their interpretation of the fractional Laplacian as Dirichlet-to-Neumann operator of a singular elliptic boundary value problem has for instance been generalized to conformally compact Einstein manifolds~\cite{CG11}.

In this work we attempt to extend their theory for functions and distributions to the case of differential forms; these are also important for physical theories -- as would also be other types of fields and vector bundles. In particular there is an interesting family of integral operators analogous to the Poisson transform with both nice analytic and geometric properties. The corresponding Dirichlet-to-Neumann operators on differential forms which play the role of the fractional Laplacian are the so-called fractional Branson--Gover operators. They do not interpolate between powers of the Laplacian on differential forms, but instead between their conformally invariant analogs, the Branson--Gover operators, which play an important role in conformal geometry.

\subsection*{A boundary value problems for differential forms}

On $\RR^n$ ($n\geq2$) we consider the standard Euclidean metric. The space $\Omega^p(\RR^n)$ of smooth $p$-forms on $\RR^n$ will be identified with $C^\infty(\RR^n)\otimes\bigwedge^p\CC^n$. In this way we can view $\calS'(\RR^n)\otimes\bigwedge^p\CC^n$ as distribution-valued $p$-forms. We write $e_i$ for the standard basis vectors in $\CC^n$ and denote by $\varepsilon_x$ and $i_x$ the exterior and interior multiplication on $\bigwedge^\bullet\CC^n$ by $x\in\CC^n$.

For $0\leq p\leq n$ and $a\in\RR$ we consider the following second order differential operator on differential $p$-forms on $\RR^n$:
\begin{equation}
 \Delta_{a,p} := x_n^2\Delta + ax_n\frac{\partial}{\partial x_n} + 2x_n(i_{e_n}d'-\varepsilon_{e_n}\delta')-(n-2p)\varepsilon_{e_n}i_{e_n},\label{eq:DefDiffOp}
\end{equation}
where $\Delta$ is the Euclidean Laplacian on differential forms and
\begin{equation}\label{eq:DefDandDelta}
 d' = \sum_{j=1}^{n-1}\varepsilon_{e_j}\frac{\partial}{\partial x_j}, \qquad \delta' = -\sum_{j=1}^{n-1}i_{e_j}\frac{\partial}{\partial x_j}
\end{equation}
are the Euclidean differential and codifferential on the subspace $\RR^{n-1}$.

The appropriate Hilbert space on which this operator acts is a \emph{homogeneous Sobolev space} which is most easily defined in terms of the Euclidean Fourier transform $\widehat{u}$ of a $p$-form $u$ (see Section~\ref{sec:CSandSobolev} for details):
$$ \dot{H}^{s,p}(\RR^n) = \left\{u\in\calS'(\RR^n)\otimes\bigwedge^p\CC^n:\int_{\RR^n}|\xi|^{2s}\|\widehat{u}(\xi)\|^2\,d\xi<\infty\right\}. $$
Note that $\dot{H}^{0,p}(\RR^n)=L^{2,p}(\RR^n)$ is the space of $L^2$-forms of degree $p$. Instead of working with the obvious norm on $\dot{H}^{s,p}(\RR^n)$, we use a slightly different but equivalent norm $\|\cdot\|_{s,p}$ which has the advantage that it is conformally invariant (see Proposition~\ref{prop:CSInnerProduct} for the precise definition).

To state the boundary value problem we remark that by the Sobolev Trace Theorem (see Corollary~\ref{cor:SobolevTraceThm}) there exists for $\frac{1}{2}<s<\frac{n}{2}$ a restriction map
$$ \dot{H}^{s,p}(\RR^n)\to\dot{H}^{s-\frac{1}{2},p}(\RR^{n-1}), \quad u\mapsto u|_{\RR^{n-1}}, $$
which agrees with the pullback by the embedding $\RR^{n-1}\hookrightarrow\RR^n$ on smooth differential forms. We further note that by duality $\bigwedge^{n-p}\RR^n\simeq\bigwedge^p\RR^n$ it suffices to consider the case $0\leq p\leq\frac{n}{2}$.

\begin{thmalph}[{see Section~\ref{sec:DiffSBOs}}]\label{thm:Dirichlet}
Assume $0\leq p\leq\frac{n}{2}$.
\begin{enumerate}
\item For $2-n+2p<a\leq2$ the operator $\Delta_{a,p}$ is essentially self-adjoint on the homogeneous Sobolev space $\dot{H}^{\frac{2-a}{2},p}(\RR^n)$ with respect to the conformally invariant norm $\|\cdot\|_{\frac{2-a}{2},p}$. Its point spectrum contains $\{k(k+a-1):k\in\NN,k<\frac{1-a}{2}\}$.
\item For $2-n+2p<a<1$ and $f\in\dot{H}^{\frac{1-a}{2},p}(\RR^{n-1})$ the Dirichlet problem
\begin{equation}
 \Delta_{a,p}u = 0, \quad u|_{\RR^{n-1}} = f\label{eq:IntroBdyValueProblem}
\end{equation}
has a unique solution $u\in\dot{H}^{\frac{2-a}{2},p}(\RR^n)$.
\end{enumerate}
\end{thmalph}

It is worth mentioning that the differential operator $\Delta_{a,p}$ is invariant under the action of the conformal group of the subspace $\RR^{n-1}\subseteq\RR^n$ (but not under the action of the conformal group of $\RR^n$).

\subsection*{The Poisson transform}

For a bounded function $f$ on $\RR^{n-1}$ we define
$$ P_{a,p}f(x) = c_{a,p}\int_{\RR^{n-1}}\frac{|x_n|^{1-a}}{(|x'-y|^2+x_n^2)^{\frac{n-a+2}{2}}}\big(i_{x-y}\varepsilon_{x-y}-\varepsilon_{x-y}i_{x-y}\big)f(y)\,dy \qquad (x\in\RR^n) $$
where $c_{a,p}=2\pi^{-\frac{n-1}{2}}\Gamma(\frac{n-a+2}{2})\Gamma(\frac{1-a}{2})^{-1}(n-2p-a)^{-1}$. The integral operator $P_{a,p}$ turns out to be the Poisson transform of the boundary value problem \eqref{eq:IntroBdyValueProblem} and it extends to an isometry (up to a scalar) between the corresponding homogeneous Sobolev spaces:

\begin{thmalph}[{see Section~\ref{sec:IntSBOs}}]\label{thm:Poisson}
Assume $0\leq p\leq\frac{n}{2}$ and $2-n+2p<a<1$.
\begin{enumerate}
\item The integral operator $P_{a,p}$ extends to a continuous linear operator
$$ P_{a,p}:\dot{H}^{\frac{1-a}{2},p}(\RR^{n-1})\to\dot{H}^{\frac{2-a}{2},p}(\RR^n), $$
which maps $f$ to the unique solution $u=P_{a,p}f$ of the boundary value problem \eqref{eq:IntroBdyValueProblem}.
\item The operator $P_{a,p}$ is isometric up to a constant. More precisely,
$$ \|P_{a,p}f\|_{\frac{2-a}{2},p}^2 = \frac{2\sqrt{\pi}(n-2p-a+2)\Gamma(\frac{2-a}{2})}{(n-2p-a)\Gamma(\frac{1-a}{2})}\cdot\|f\|_{\frac{1-a}{2},p}^2. $$
\end{enumerate}
\end{thmalph}

In Section~\ref{sec:Uniqueness} we even find an explicit formula for the Fourier transform $\widehat{P_{a,p}f}$ of $P_{a,p}f$ in terms of the Fourier transform $\widehat{f}$ of $f$ (see Theorem~\ref{thm:FTofPoisson}). In fact, this formula is used to find the precise constant in the isometry property.

\subsection*{Fractional Branson--Gover operators}

We finally identify the \emph{Dirichlet-to-Neumann map} of the boundary value problem \eqref{eq:IntroBdyValueProblem} as a fractional Branson--Gover operator. For this we consider, instead of the ordinary powers of the Laplacian on differential forms on $\RR^{n-1}$
$$ \Delta^N = (\delta d+d\delta)^N = (\delta d)^N+(d\delta)^N $$
the conformally invariant operators
$$ D_{N,p} = \Big(\frac{n-1}{2}-p+N\Big)(\delta d)^N+\Big(\frac{n-1}{2}-p-N\Big)(d\delta)^N $$
found by Branson--Gover~\cite{BG05}. These operators play an important role in conformal geometry since their construction can be generalized to conformal manifolds. The Branson--Gover operators are interpolated by the \emph{fractional Branson--Gover operators}
$$ L_{s,p}\omega(x) = \frac{1}{\Gamma(-s)}\int_{\RR^n} |y|^{-2s-n-2}(i_y\varepsilon_y-\varepsilon_yi_y)\big[\omega(x+y)-\omega(x)\big]\,dy, $$
which are also conformally invariant. More precisely, for $N\in\NN$ we have (see Fischmann--{\O}rsted~\cite[Corollary 4.6]{FO17})
$$ L_{N,p} = \frac{\pi^{\frac{n}{2}}}{4^N\Gamma(\frac{n}{2}+N+1)}D_{N,p}. $$
In analogy to the work of Caffarelli--Silvestre~\cite{CS07} for the fractional powers $\Delta^s$ of the scalar Laplacian, the fractional Branson--Gover operators $L_{s,p}$ can be interpreted as the Dirichlet-to-Neumann map of the boundary value problem \eqref{eq:IntroBdyValueProblem}:

\begin{thmalph}\label{thm:DtoN}
Assume $0\leq p\leq\frac{n-3}{2}$ and let $s\in(0,1)$ and $a=1-2s$. For $f\in\dot{H}^{\frac{1-a}{2},p}(\RR^{n-1})$ let $u=P_{a,p}f\in\dot{H}^{\frac{2-a}{2},p}(\RR^n)$ be the unique solution of \eqref{eq:IntroBdyValueProblem}. Then
$$ L_{s,p}f(y) = d_{s,p}\lim_{x_n\to0} x_n^a \partial_{x_n}u(x',x_n) $$
with $d_{s,p}=(\Gamma(-s)c_{a,p})^{-1}$.
\end{thmalph}

Since the operator $\Delta_{a,p}$ is also invariant under the conformal group of $\RR^{n-1}$, the description of the fractional Branson--Gover operators as Dirichlet-to-Neumann maps of the boundary value problems \eqref{eq:IntroBdyValueProblem} respects the action of the conformal group of $\RR^{n-1}$.

\subsection*{Methods}

Most of our proofs rely on the representation theory of the conformal group ${\rm O}(1,n+1)$ of $\RR^n$ and its subgroup ${\rm O}(1,n)$. More precisely, the group ${\rm O}(1,n+1)$ acts on the homogeneous Sobolev space $\dot{H}^{s,p}(\RR^n)$, $-\frac{n}{2}+p<s<\frac{n}{2}-p$, by an irreducible unitary representation (the complementary series). Restricted to the subgroup ${\rm O}(1,n)\subseteq{\rm O}(1,n+1)$, the representation decomposes into irreducible unitary representations of ${\rm O}(1,n)$. For $\frac{1}{2}<s<\frac{n}{2}-p$ one of these representations is the corresponding complementary series representation of ${\rm O}(1,n)$ on $\dot{H}^{s-\frac{1}{2},p}(\RR^{n-1})$, and the restriction map $\dot{H}^{s,p}(\RR^n)\to\dot{H}^{s-\frac{1}{2},p}(\RR^{n-1})$ projects onto this component. This observation makes it possible to use the machinery of \emph{symmetry breaking operators} whose study was recently initiated by Kobayashi~\cite{Kob15} (see also \cite{FJS16,KKP16,KS18,MO17}). In this language the differential operator $\Delta_{a,p}$ ($a=2(1-s)$) corresponds to the action of the Casimir element of ${\rm O}(1,n)$ in $\dot{H}^{s,p}(\RR^n)$ and the fractional Branson--Gover operators are the standard Knapp--Stein intertwining operators between principal series representations of the group ${\rm O}(1,n)$.

\subsection*{Structure of the paper}

In Section~\ref{sec:CSReps} we briefly recall the action of the conformal group ${\rm O}(1,n+1)$ on $\RR^n$ and the corresponding unitary representations on homogeneous Sobolev spaces of differential forms on $\RR^n$, the complementary series representations. Here we also give a representation theoretic interpretation of the fractional Branson--Gover operators as intertwining operators between complementary series representations. In Section~\ref{sec:SBO} the relation between symmetry breaking operators in representation theory and boundary value problems and Poisson transforms is established. Here most of the statements in Theorem~\ref{thm:Dirichlet} and \ref{thm:Poisson} are proven. The remaining points are addressed in Section~\ref{sec:ExplicitValues} (explicit normalization of the integral formula for $P_{a,p}$), Section~\ref{sec:Uniqueness} (uniqueness of solutions to \eqref{eq:IntroBdyValueProblem}) and Section~\ref{sec:Isometry} (isometry property of $P_{a,p}$). Finally, in Section~\ref{sec:DtoN} the fractional Branson--Gover operators are identified with the Dirichlet-to-Neumann map of the boundary value problem \eqref{eq:IntroBdyValueProblem}, providing a proof of Theorem~\ref{thm:DtoN}.

In Appendix~\ref{sec:Casimir} we further give some computational details related to the interpretation of $\Delta_{a,p}$ as a Casimir operator, and in Appendix~\ref{sec:Laplacian} we compare $\Delta_{a,p}$ with the Laplace--Beltrami operator on differential forms on the hyperbolic space realized as the upper half space $\HH^n\subseteq\RR^n$.

\section{Action of the conformal group on differential forms}\label{sec:CSReps}

In this section we sketch the construction of the complementary series representations of the conformal group $G={\rm O}(1,n+1)$ on differential forms.

\subsection{The conformal group}

We realize the rank one orthogonal group $G={\rm O}(1,n+1)$, $n\geq1$, as $(n+2)\times(n+2)$ matrices preserving the bilinear form
$$ (x,y) \mapsto x_0y_0-x_1y_1-\cdots-x_{n+1}y_{n+1}. $$
Let $\frakg$ denote the Lie algebra of $G$ and define
$$ H := \left(\begin{array}{ccc}0&1&\\1&0&\\&&\0_n\end{array}\right)\in\frakg. $$
Then the adjoint action $\ad(H)$ on $\frakg$ has eigenvalues $+1$, $0$ and $-1$ and we write $\frakn$, $\frakl$ and $\overline{\frakn}$ for the respective eigenspaces which are in fact subalgebras. The subalgebra $\frakl$ can be further decomposed as $\frakl=\frakm\oplus\fraka$ with $\frakm$ the Lie algebra of
$$ M := \left\{\begin{pmatrix}\varepsilon&&\\&\varepsilon&\\&&m\end{pmatrix}:\varepsilon=\pm1,m\in{\rm O}(n)\right\} \simeq {\rm O}(1)\times{\rm O}(n) $$
and $\fraka=\RR H$. We further write $A=\exp(\fraka)$, $N=\exp(\frakn)$ and $\overline{N}=\exp(\overline{\frakn})$, then $P=MAN$ and $\overline{P}=MA\overline{N}$ are parabolic subgroups of $G$. They are conjugate via the element $w_0=\diag(-1,1,\ldots,1)\in G$, i.e. $w_0Pw_0^{-1}=\overline{P}$. In what follows we identify $\overline{N}\simeq\RR^n$ by
\begin{equation}
 \RR^n\to\overline{N}, \quad x\mapsto \overline{n}_x := \exp\left(\begin{array}{ccc}0&0&x^\top\\0&0&-x^\top\\x&x&\0_n\end{array}\right).\label{eq:IdentificationNbar}
\end{equation}

The group $G$ acts by rational conformal transformations on $\RR^n$ in the following way: The subset $\overline{N}MAN\subseteq G$ is open and dense, so that for fixed $g\in G$ and almost all $x\in\RR^n$ we can decompose
$$ g\overline{n}_x = \overline{n}_{g\cdot x}m(g,x)e^{-\log(j(g,x))H}n $$
with $g\cdot x\in\RR^n$, $m(g,x)\in M$, $j(g,x)>0$ and $n\in N$. This defines a rational conformal action $(g,x)\mapsto g\cdot x$ of $G$ on $\RR^n$ with conformal factor $j(g,x)$ in the sense that the derivative $Dg(x)$ of $g$ at $x\in\RR^n$ satisfies
$$ |Dg(x)\xi| = j(g,x)|\xi| \qquad \forall\,\xi\in\RR^n. $$

\subsection{Principal series representations on differential forms}

We identify $\fraka_\CC^*\simeq\CC$ by $\lambda\mapsto\lambda(H)$. Then the half sum of positive roots $\rho:=\frac{1}{2}\tr\ad|_{\frakn}\in\fraka^*$ is given by $\rho=\frac{n}{2}$. For $\lambda\in\fraka_\CC^*\simeq\CC$ the character $e^\lambda$ of $A$ is given by $e^\lambda(e^{tH})=e^{\lambda t}$.

For an irreducible representation $(\xi,V)$ of $M$ and $\lambda\in\CC$ we define the principal series representation (smooth normalized parabolic induction)
$$ \pi_{\lambda,\xi}^\infty := \Ind_P^G(\xi\otimes e^\lambda\otimes\1) $$
as the representation of $G$ on the Fréchet space
$$ \{u\in C^\infty(G,V):u(gman)=a^{-\lambda-\rho}\xi(m)^{-1}u(g)\,\forall\,g\in G,man\in MAN\} $$
by left-translation, i.e. $\pi_{\lambda,\xi}^\infty(g)u(x)=u(g^{-1}x)$, $g,x\in G$.

We will mostly work in a different and more convenient realization of these representations, the non-compact picture, which we briefly explain. The subset $\overline{N}MAN\subseteq G$ is open and dense and therefore, restriction to $\overline{N}\simeq\RR^n$ realizes the representation $\pi_{\lambda,\xi}^\infty$ on a space
$$ I_{\lambda,\xi}^\infty \subseteq C^\infty(\RR^n,V) $$
of smooth $V$-valued functions on $\RR^n$.

Here we are mostly interested in the case where $\xi$ is the $p$-th exterior power of the standard representation of ${\rm O}(n)$ on $\CC^n$, i.e. $V=\bigwedge^p\CC^n$, $0\leq p\leq n$. We denote this action by $\xi_p$ and extend it trivially to the group $M\simeq{\rm O}(1)\times{\rm O}(n)$. We write $\pi_{\lambda,p}^\infty=\pi_{\lambda,\xi_p}^\infty$ and $I_{\lambda,p}^\infty=I_{\lambda,\xi_p}^\infty$ for short. Identifying $e_i$ with $dx_i$ the space $I_{\lambda,p}^\infty$ can be viewed as subspace of the space $\Omega^p(\RR^n)$ of differential $p$-forms on $\RR^n$. We note that $I_{\lambda,p}^\infty$ always contains the space $\calS(\RR^n,\bigwedge^p\CC^n)$ of rapidly decreasing $p$-forms.

On $\bigwedge^p\CC^n$ we use the standard inner product so that $\{e_{i_1}\wedge\ldots\wedge e_{i_p}:1\leq i_1<\ldots<i_p\leq n\}$ forms an orthonormal basis. With respect to this inner product the representation $\xi_p$ of $M$ on $\bigwedge^p\CC^n$ is unitary. Moreover, the inner product can be used to define a $G$-invariant continuous bilinear pairing
$$ I_{\lambda,p}^\infty\times I_{-\lambda,p}^\infty \to \CC, \quad (u_1,u_2)\mapsto \int_{\RR^n}\langle u_1(x),u_2(x)\rangle\,dx $$
and hence identify $I_{\lambda,p}^\infty$ with a subspace of the dual space $I_{\lambda,p}^{-\infty}:=(I_{-\lambda,p}^\infty)^*$. We have
$$ \calS(\RR^n,\bigwedge^p\CC^n)\subseteq I_{\lambda,p}^\infty\subseteq I_{\lambda,p}^{-\infty}\subseteq\calS'(\RR^n,\bigwedge^p\CC^n) $$
and the representation $\pi_{\lambda,p}^\infty$ extends by duality to a representation $\pi_{\lambda,p}^{-\infty}$ on $I_{\lambda,p}^{-\infty}$. In terms of the conformal action of $G$ on $\RR^n$ the representation is given by
$$ \pi_{\lambda,p}(g)u(x) = j(g^{-1},x)^{\lambda+\rho}\xi_p(m(g^{-1},x))^{-1}u(g^{-1}\cdot x) \qquad \forall\,g\in G,x\in\RR^n. $$

\subsection{Knapp--Stein intertwining operators}

There exists a meromorphic family of intertwining operators $T_{\lambda,p}:\pi_{\lambda,p}^\infty\to\pi_{-\lambda,p}^\infty$, the so-called Knapp--Stein intertwiners. For $\Re\lambda>0$ the operator $T_{\lambda,p}$ is given by the convergent integral
$$ T_{\lambda,p} u(g) = \int_{\overline{N}} u(gw_0\overline{n})\,d\overline{n}. $$
Abusing notation we also write $T_{\lambda,p}$ for the corresponding operator $I_{\lambda,p}^\infty\to I_{-\lambda,p}^\infty$. In \cite{FO17,SV11} the following expression for $T_{\lambda,p}$ as an integral kernel operator was obtained:

\begin{lemma}
Let $0\leq p\leq n$. For $\Re\lambda>0$ the Knapp--Stein intertwining operator is given by
$$ T_{\lambda,p}u(x) = \int_{\mathbb R^n} |y|^{2(\lambda-\rho-1)} (i_{y}\epsilon_y - \epsilon_y i_y )u(x+y)\,dy. $$
\end{lemma}

Note that with the notation $y=|y|\widehat y$ the operator $T_{\lambda,p}$ can also be written as
$$ T_{\lambda,p}u(x) = \int_{\mathbb R^n} |y|^{2(\lambda-\rho)} (i_{\widehat y}\epsilon_{\widehat y} - \epsilon_{\widehat y} i_{\widehat y} )u(x+y)\,dy. $$

Of particular importance for us are the Knapp--Stein intertwiners $T_{\lambda,p}$ for $-1<\lambda<0$, so we describe their regularization in detail. Let $u\in\calS(\RR^n,\bigwedge^p\CC^n)$, then for $\Re\lambda>0$ we can write
$$ T_{\lambda,p}u(x) = \int_0^\infty r^{2\lambda-1}\widetilde{u}(x,r)\,dr \qquad \mbox{with} \quad \widetilde{u}(x,r) = \int_{S^{n-1}}(i_\omega\varepsilon_\omega-\varepsilon_\omega i_\omega)u(x+r\omega)\,d\omega. $$
Note that $\widetilde{u}(x,r)$ is an even function of $r$, i.e. $\widetilde{u}(x,-r)=\widetilde{u}(x,r)$. Now the standard regularization for the distributions $|r|^{2\lambda-1}$ on $\RR$ gives for $\lambda\in(-1,0)$:
\begin{align*}
 T_{\lambda,p}u(x) &= \PV \int_0^\infty r^{2\lambda-1}(\widetilde{f}(x,r)-\widetilde{f}(x,0))\,dr\\
 &= \PV \int_0^\infty r^{2\lambda-1}\int_{S^{n-1}} (i_\omega\varepsilon_\omega-\varepsilon_\omega i_\omega)(f(x+r\omega)-f(x))\,d\omega\,dr\\
 &= \PV \int_{\RR^n} |y|^{2\lambda-n}(i_{\widehat{y}}\varepsilon_{\widehat{y}}-\varepsilon_{\widehat{y}}i_{\widehat{y}})(f(x+y)-f(x))\,dy.
\end{align*}

\subsection{Complementary series representations and homogeneous Sobolev spaces}\label{sec:CSandSobolev}

Let $0\leq p\leq n$. For $\lambda\in i\RR$ the representation $\pi_{\lambda,p}^\infty$ is irreducible except for the case $(\lambda,p)=(0,\frac{n}{2})$ with $n$ even where it decomposes into the direct sum of two irreducible representations. For all $\lambda\in i\RR$ the representation $\pi_{\lambda,p}^\infty$ extends to a unitary representation on $L^2(\RR^n,\bigwedge^p\CC^n)$ which we interpret as the space $L^{2,p}(\RR^n)$ of $L^2$-forms of degree $p$.

More subtle is the question about unitarizability for $\lambda\in\RR$. For simplicity we assume $0\leq p\leq\frac{n}{2}$, the remaining cases can be treated similarly. It turns out that $\pi_{\lambda,p}^\infty$ is irreducible and unitarizable if and only if $|\lambda|<\frac{n}{2}-p$. In this case the $G$-invariant norm on $I_{\lambda,p}^\infty$ is given by
\begin{equation}
 \|u\|_\lambda^2 = \int_{\RR^n}\langle T_{\lambda,p}u(x),u(x)\rangle\,dx\label{eq:CSNormReal}
\end{equation}
for $\lambda\in(0,\frac{n}{2}-p)$ and by a regularization of the integral in the remaining cases. We write $I_{\lambda,p}$ for the corresponding Hilbert space and extend $\pi_{\lambda,p}^\infty$ to an irreducible unitary representation $\pi_{\lambda,p}$ on $I_{\lambda,p}$, the \emph{complementary series}. The smooth vectors of this representation are given by $I_{\lambda,p}^\infty$ and we have the following inclusions:
$$ \calS(\RR^n,\bigwedge^p\CC^n)\subseteq I_{\lambda,p}^\infty\subseteq I_{\lambda,p}\subseteq I_{\lambda,p}^{-\infty}\subseteq\calS'(\RR^n,\bigwedge^p\CC^n). $$

A convenient way to handle the regularization of the integral is by taking the Euclidean Fourier transform. We use the following normalization:
$$ \widehat{u}(\xi) = (2\pi)^{-\frac{n}{2}}\int_{\RR^n} e^{-ix\cdot\xi}u(x)\,dx, \qquad u\in\calS'(\RR^n,\bigwedge^p\CC^n). $$
In \cite[Corollary 4.2, Remark 4.10]{FO17} the following equivalent description of the invariant norm is given:

\begin{proposition}\label{prop:CSInnerProduct}
For $|\lambda|<\frac n2 -p$ the $G$-invariant norm on $I_{\lambda,p}$ is given by
\begin{equation}\label{norm-1}
\begin{split}
 \Vert u\Vert^2_{\lambda} &= \int_{\mathbb R^n} |\xi|^{-2\lambda-2}\left\langle \Big( \Big(\frac{n}2 - p - \lambda\Big) i_\xi\epsilon_\xi + \Big(\frac n2 - p +\lambda\Big)\epsilon_\xi i_\xi \Big)\widehat u(\xi), \widehat u(\xi)\right\rangle\,d\xi\\
 &= \int_{\mathbb R^n} |\xi|^{-2\lambda}\left\langle \Big( \Big(\frac{n}2 - p - \lambda\Big) i_{\widehat \xi}\epsilon_{\widehat \xi} + \Big(\frac n2 - p +\lambda\Big)\epsilon_{\widehat \xi} i_{\widehat \xi} \Big)\widehat u(\xi), \widehat u(\xi)\right\rangle\,d\xi.
\end{split}
\end{equation}
\end{proposition}

Sometimes it is more convenient to work with an equivalent norm which is not $G$-invariant but easier to handle:

\begin{lemma}
For $|\lambda|<\frac{n}{2}-p$ the norm
\begin{equation}\label{norm-var}
 |u|_{\lambda}^2:=\int_{\mathbb R^n}|\xi|^{-2\lambda}\Vert\widehat{u}(\xi)\Vert^2\,d\xi
\end{equation}
is equivalent to the norm $\|\!\cdot\!\|_\lambda$ in \eqref{norm-1}. More precisely,
$$ \Big(\frac{n}2 - p - |\lambda|\Big) |u|_{\lambda}^2 \le \Vert u\Vert^2_{\lambda} \le \Big(\frac n2 - p +|\lambda|\Big) |u|_{\lambda}^2. $$
\end{lemma}

\begin{proof}
For any unit vector $u$ we have $i_u\varepsilon_u \ge 0$, $\varepsilon_u  i_u \ge 0$ and $i_u\varepsilon_u+\varepsilon_u  i_u=\id$ as operators on $\bigwedge^p\mathbb C^n$. Thus
$$ \Big(\frac{n}2 - p - |\lambda|\Big) \id \le \Big(\frac{n}2 - p - \lambda\Big) i_{\widehat \xi}\epsilon_{\widehat \xi} + \Big(\frac n2 - p +\lambda\Big) \epsilon_{\widehat \xi} i_{\widehat \xi} \le \Big(\frac n2 - p + |\lambda|\Big) \id $$
for $|\lambda|<\frac n2 -p$. The claimed estimate now follows by integration.
\end{proof}

The previous lemma shows that for $p=0$ the Hilbert space $I_{\lambda,0}$ equals the \emph{homogeneous Sobolev space}
$$ \dot{H}^s(\RR^n) = \left\{u\in\calS'(\RR^n):\int_{\RR^n}|\xi|^{2s}|\widehat{u}(\xi)|^2\,d\xi<\infty\right\} $$
of degree $s=-\lambda$. We therefore call $I_{\lambda,p}$ the \emph{homogeneous Sobolev space of $p$-forms on $\RR^n$ of degree $s$}:
$$ \dot{H}^{s,p}(\RR^n) = \left\{u\in\calS'(\RR^n,\bigwedge^p\CC^n):\int_{\RR^n}|\xi|^{2s}\|\widehat{u}(\xi)\|^2\,d\xi<\infty\right\}. $$

\section{Symmetry breaking, boundary value problems and Poisson transforms}\label{sec:SBO}

We recall the construction of symmetry breaking operators for differential forms from \cite{KS18,MO17} and describe their relation to boundary value problems. Analogous results in the scalar case were obtained in \cite[Section 3]{MOZ16a}. Although the proofs in this section resemble those in \cite{MOZ16a}, we include them for the sake of completeness.

\subsection{The subgroup ${\rm O}(1,n)$ and its representations}

The conformal group $G'={\rm O}(1,n)$ of the subspace $\RR^{n-1}\subseteq\RR^n$ can be embedded as a subgroup of $G={\rm O}(1,n+1)$ as the upper left corner. Then $P'=P\cap G'$ is a parabolic subgroup of $G'$ with Langlands decomposition $P'=M'AN'$, where $M'=M\cap G'\simeq{\rm O}(1)\times{\rm O}(n-1)$ and $N'=N\cap G'\simeq\RR^{n-1}$. Under the identification $\overline{N}\simeq\RR^n$ the subgroup $\overline{N}'=\overline{N}\cap G'$ corresponds to the subspace $\RR^{n-1}\simeq\{(x',0):x'\in\RR^{n-1}\}\subseteq\RR^n$.

For $0\leq q\leq n-1$ we let $\eta_q$ denote the representation of $M'$ on $\bigwedge^q\CC^{n-1}$. As above, we consider for $\nu\in\CC$ the principal series representations
$$ \tau_{\nu,q}^\infty := \Ind_{P'}^{G'}(\eta_q\otimes e^\nu\otimes\1) $$
of $G'$. Again we realize these representations on a space $J_{\nu,q}^\infty$ of smooth differential $q$-forms on $\overline{N}'\simeq\RR^{n-1}$. The dual space $J_{\nu,q}^{-\infty}:=(J_{-\nu,q}^\infty)^*$ will be identified with a space of distributional $q$-forms on $\RR^{n-1}$ on which $G'$ acts via duality by a representation $\tau_{\nu,q}^{-\infty}$.

For $|\nu|<\frac{n-1}{2}-q$ the representation $\tau_{\nu,q}^\infty$ is irreducible and unitarizable and we write $J_{\nu,q}$ for the corresponding Hilbert space completion of $J_{\nu,q}^\infty$ and $\tau_{\nu,q}$ for the extension of $\tau_{\nu,q}^\infty$ to $J_{\nu,q}$. As before we have $J_{\nu,q}=\dot{H}^{-\nu,q}(\RR^{n-1})$.

\subsection{The Casimir operator}\label{sec:CasimirOperator}

On the Lie algebra $\frakg$ the Killing form
$$ B(X,Y) := \frac{1}{2}\tr(XY), \qquad X,Y\in\frakg, $$
is non-degenerate, bilinear and $G$-invariant, and it restricts to a non-degenerate bilinear form on the Lie algebra $\frakg'$ of $G'$. Let $(X_\alpha)_\alpha\subseteq\frakg'$ be a basis of $\frakg'$ and let $(\widehat{X}_\alpha)_\alpha$ be its dual basis with respect to the form $B$. Then the \emph{Casimir element}
$$ C = \sum_\alpha X_\alpha\widehat{X}_\alpha\in\calU(\frakg) $$
in the universal enveloping algebra of $\frakg$ is independent of the chosen basis and invariant under $\Ad(G')$. We study the action of $C$ in the representation $\pi_{\lambda,p}^{\infty}$. For this denote by $d\pi_{\lambda,p}^{\infty}$ the derived representation of $\calU(\frakg)$ on $I_{\lambda,p}^{\infty}$.

\begin{proposition}\label{prop:ActionCasimir}
For $0\leq p\leq n$ and $\lambda\in\CC$ we have
$$ d\pi_{\lambda,p}^{\infty}(C) = \Delta_{2(\lambda+1),p} + (\lambda+\rho)(\lambda-\rho+1) + p(n-p-1), $$
where $\Delta_{a,p}$ denotes the differential operator defined in \eqref{eq:DefDiffOp} and $d'$ and $\delta'$ are the differential and codifferential on $\RR^{n-1}$ defined in \eqref{eq:DefDandDelta}.
\end{proposition}

\begin{proof}
By the computation \eqref{eq:FormulaCasimirReal} in Appendix~\ref{sec:Casimir} we have
\begin{multline*}
 d\pi_{\lambda,p}^{\infty}(C) = x_n^2\Delta+2(\lambda+1)x_n\frac{\partial}{\partial x_n}-2x_n\sum_{j=1}^{n-1}d\xi_p(M_{jn})\frac{\partial}{\partial x_j}\\
 -\sum_{1\leq j<k\leq n-1}d\xi_p(M_{jk})^2+(\lambda+\rho)(\lambda-\rho+1).
\end{multline*}
Now we first note that $d\xi_p(M_{jn})=-(i_{e_n}\varepsilon_{e_j}+\varepsilon_{e_n}i_{e_j})$, then the first sum can be computed with \eqref{eq:DefDandDelta}:
$$ -2x_n\sum_{j=1}^{n-1}d\xi_p(M_{jn})\frac{\partial}{\partial x_j} = 2x_n(i_{e_n}d'-\varepsilon_{e_n}\delta'). $$
Further, the expression $\sum_{1\leq j<k\leq n-1}d\xi^{(p)}(M_{jk})^2$ is simply the Casimir operator of $\so(n-1)$ acting on $\bigwedge^p\CC^n$. The irreducible representation $\bigwedge^p\CC^n$ of $\so(n)$ decomposes into two irreducible summands when restricted to $\so(n-1)$, namely $\bigwedge^p\CC^{n-1}$ and $\bigwedge^{p-1}\CC^{n-1}\wedge e_n$. The projection onto $\bigwedge^p\CC^{n-1}$ is given by $i_{e_n}\varepsilon_{e_n}$ and the projection onto $\bigwedge^{p-1}\CC^{n-1}\wedge e_n$ is given by $\varepsilon_{e_n}i_{e_n}$. Moreover, the Casimir element of $\so(n-1)$ acts on $\bigwedge^q\CC^{n-1}$ by the scalar $-q(n-q-1)$ ($q=p-1,p$), so that
$$ \sum_{1\leq j<k\leq n-1}d\xi_p(M_{jk})^2 = -p(n-p-1)i_{e_n}\varepsilon_{e_n}-(p-1)(n-p)\varepsilon_{e_n}i_{e_n}. $$
Using $i_{e_n}\varepsilon_{e_n}+\varepsilon_{e_n}i_{e_n}=\id$ finally yields the claimed formula.
\end{proof}

\subsection{Differential symmetry breaking operators and boundary value problems}\label{sec:DiffSBOs}

The restriction of the irreducible representation $\pi_{\lambda,p}^\infty$ of $G$ on $I_{\lambda,p}^\infty\subseteq\Omega^p(\RR^n)$ to the subgroup $G'\subseteq G$ defines a representation $\pi_{\lambda,p}^\infty|_{G'}$ of $G'$ which is highly reducible. The irreducible representations of $G'$ which occur inside $\pi_{\lambda,p}^\infty|_{G'}$ are described in terms of so-called \emph{symmetry breaking operators} (see e.g. Kobayashi~\cite{Kob15}). In our setting, a continuous linear operator $T:I_{\lambda,p}^\infty\to J_{\nu,q}^\infty$ is called \emph{symmetry breaking operator} if $T$ intertwines the representations $\pi_{\lambda,p}^\infty|_{G'}$ and $\tau_{\nu,q}^\infty$:
$$ T\circ\pi_{\lambda,p}^\infty(g) = \tau_{\nu,q}^\infty(G)\circ T \qquad \forall\,g\in G'. $$

The symmetry breaking operators between $\pi_{\lambda,p}^\infty$ and $\tau_{\nu,q}^\infty$ were classified by Kobayashi--Speh~\cite{KS18}. Of particular importance for us are \emph{differential symmetry breaking operators}. In our special case these are symmetry breaking operators which arise as the composition of a differential operator on $\Omega^p(\RR^n)$ and the restriction from $\Omega^p(\RR^n)\to\Omega^p(\RR^{n-1})$. Differential symmetry breaking operators between differential forms were classified by Fischmann--Juhl--Somberg~\cite{FJS16} and Kobayashi--Kubo--Pevzner~\cite{KKP16}, and their classification contains one particular family of operators which is important for our purpose:

\begin{theorem}[{\cite[Theorem~1.6~(1)]{KKP16}}]
Suppose $\lambda+\rho-\nu-\rho'=-2k$ for some integer $k\ge 0$, then there exists a non-trivial differential symmetry breaking operator
$$ C_{\lambda,\nu,p}:I_{\lambda,p}^\infty\to J_{\nu,p}^\infty $$
which is of the form
$$ C_{\lambda,\nu,p}u(x') = (P_{\lambda,\nu}u)(x',0) $$
where $P_{\lambda,\nu}=p_{\lambda,\nu}(\frac{\partial}{\partial x_1},\ldots,\frac{\partial}{\partial x_n})$ for a $\Hom_\CC(\bigwedge^p\CC^n,\bigwedge^p\CC^{n-1})$-valued homogeneous polynomial $p_{\lambda,\nu}(\xi_1,\ldots,\xi_n)$ of degree $2k$.
\end{theorem}

We remark that for $k=0$ the polynomial $p_{\lambda,\nu}$ is constant, so that the operator $C_{\lambda,\nu,p}$ is (up to scaling) the restriction of differential forms on $\RR^n$ to $\RR^{n-1}$.

\begin{theorem}\label{thm:DiffSBOsContinuous}
Assume $0\leq p\leq\frac{n}{2}$ and suppose that $\lambda+\rho-\nu-\rho'=-2k$, $k\geq0$. If $\lambda\in(-\frac{n}{2}+p,0)$ and $\nu\in(-\frac{n-1}{2}+p,0)$, then the differential symmetry breaking operator $C_{\lambda,\nu,p}$ extends to a non-trivial continuous linear operator between the Hilbert spaces $I_{\lambda,p}$ and $J_{\nu,p}$:
$$ C_{\lambda,\nu,p}:I_{\lambda,p}\to J_{\nu,p}. $$
\end{theorem}

\begin{proof}
We write $R$ for the restriction operator $Ru(x')=u(x',0)$, so that $C_{\lambda,\nu,p}=R\circ P_{\lambda,\nu}$. The Fourier inversion formula shows that
$$ \widehat{Ru}(\xi') = (2\pi)^{-\frac{1}{2}} \int_\RR \widehat{u}(\xi',\xi_n)\,d\xi_n. $$
This implies
$$ \widehat{C_{\lambda,\nu,p}u}(\xi') = (2\pi)^{-\frac{1}{2}} \int_\RR \widehat{P_{\lambda,\nu}u}(\xi',\xi_n)\,d\xi_n = (2\pi)^{-\frac{1}{2}} \int_\RR p_{\lambda,\nu}(i\xi_1,\ldots,i\xi_n)\widehat{u}(\xi',\xi_n)\,d\xi_n. $$
Since $p_{\lambda,\nu}$ is homogeneous of degree $2k$ its matrix norm can be estimated by
$$ \|p_{\lambda,\nu}(i\xi_1,\ldots,i\xi_n)\| \leq C\cdot|\xi|^{2k} $$
for some constant $C>0$, whence
$$ \|\widehat{C_{\lambda,\nu,p}u}(\xi')\| \leq C\int_\RR|\xi|^{2k}\|\widehat{u}(\xi',\xi_n)\|\,d\xi_n = C\int_\RR|\xi|^{\lambda+2k}|\xi|^{-\lambda}\|\widehat{u}(\xi',\xi_n)\|\,d\xi_n. $$
Applying the Cauchy--Schwartz inequality gives
$$ \|\widehat{C_{\lambda,\nu,p}u}(\xi')\|^2 \leq C^2\int_\RR|\xi|^{2\lambda+4k}\,d\xi_n\int_\RR|\xi|^{-2\lambda}\|\widehat{u}(\xi',\xi_n)\|^2\,d\xi_n. $$
The first integral can be computed using the substitution $\xi_n=|\xi'|t$:
$$ \int_\RR|\xi|^{2\lambda+4k}\,d\xi_n = |\xi'|^{2\lambda+4k+1}\int_\RR(1+t^2)^{\lambda+2k}\,dt = C'|\xi'|^{2\nu} $$
with $C'=\int_\RR(1+t^2)^{\lambda+2k}\,dt<\infty$ since $\lambda+2k=\nu-\frac{1}{2}<-\frac{1}{2}$. Hence we obtain
$$ |\xi'|^{-2\nu}\|\widehat{C_{\lambda,\nu,p}u}(\xi')\| \leq C^2C'\int_\RR|\xi|^{-2\lambda}\|\widehat{u}(\xi',\xi_n)\|^2\,d\xi_n, $$
so that integration over $\xi'\in\RR^{n-1}$ finally shows that
\begin{equation*}
 |C_{\lambda,\nu,p}u|_\nu^2 \leq C^2C'|u|_\lambda^2.\qedhere
\end{equation*}
\end{proof}

\begin{corollary}\label{cor:SobolevTraceThm}
Assume that $0\leq p\leq\frac{n}{2}$. Then for $\frac{1}{2}<s<\frac{n}{2}-p$ the restriction $u\mapsto u|_{\RR^{n-1}}$ of compactly supported smooth $p$-forms on $\RR^n$ to $\RR^{n-1}$ extends to a continuous linear map
$$ R:\dot{H}^{s,p}(\RR^n)\to\dot{H}^{s-\frac{1}{2},p}(\RR^{n-1}). $$
\end{corollary}

\begin{proof}
Let $\lambda=-s$ and $\nu=\frac{1}{2}-s$, then $C_{\lambda,\nu,p}$ is up to a scalar multiple the restriction operator. Now the result follows from Theorem~\ref{thm:DiffSBOsContinuous} since $I_{\lambda,p}=\dot{H}^{-\lambda,p}(\RR^n)$ and $J_{\nu,p}=\dot{H}^{-\nu,p}(\RR^{n-1})$.
\end{proof}

Using the differential symmetry breaking operators $C_{\lambda,\nu,p}$ we can show that certain complementary series representations $\tau_{\nu,p}$ of $G'$ occur as direct summands inside the restriction $\pi_{\lambda,p}|_{G'}$ of a complementary series representation of $G$ to $G'$:

\begin{corollary}\label{cor:DiscreteSummands}
The adjoint operator $C_{\lambda,\nu,p}^*:J_{\nu,p}\to I_{\lambda,p}$ is a $G'$-equivariant isometry (up to a scalar) and identifies $\tau_{\nu,p}$ with a subrepresentation of $\pi_{\lambda,p}|_{G'}$. In particular, the Casimir operator $d\pi_{\lambda,p}^\infty(C)$ acts on the image $C_{\lambda,\nu,p}^*(J_{\nu,p})$ by the scalar $\nu^2-\rho'^2+p(n-p-1)$ and the composition $C_{\lambda,\nu,p}\circ C_{\lambda,\nu,p}^*:J_{\nu,p}\to J_{\nu,p}$ is a scalar multiple of the identity.
\end{corollary}

\begin{proof}
Since $C_{\lambda,\nu,p}$ is $G'$-intertwining, its adjoint $C_{\lambda,\nu,p}^*$ is $G'$-intertwining as well. Now $\tau_{\nu,p}$ is irreducible and therefore, by Schur's Lemma, the intertwiner $C_{\lambda,\nu,p}^*$ has to be a scalar multiple of an isometry which proves the first statement. To prove the second statement we observe that, as a parabolically induced representation, $\tau_{\nu,p}$ has infinitesimal character $\nu$ plus the infinitesimal character of $\bigwedge^p\CC^{n-1}$. Therefore the Casimir element $C$ acts by $d\tau_{\nu,p}^\infty(C)=\nu^2-\rho'^2+p(n-p-1)$. Since $C_{\lambda,\nu,p}^*$ is $G'$-intertwining, the Casimir element acts by the same scalar on the image $C_{\lambda,\nu,p}^*(J_{\nu,p})$. Finally, the composition $C_{\lambda,\nu,p}\circ C_{\lambda,\nu,p}^*:J_{\nu,p}\to J_{\nu,p}$ is a $G'$-intertwining operator from the irreducible representation $J_{\nu,p}$ to itself and hence a scalar multiple of the identity by Schur's Lemma.
\end{proof}

\begin{proof}[Proof of Theorem~\ref{thm:Dirichlet}]
Let $\lambda=\frac{a-2}{2}\in(-\frac{n}{2}+p,0)$. By Proposition~\ref{prop:ActionCasimir} the operator $\Delta_{a,p}=\Delta_{2(\lambda+1),p}$ only differs from $d\pi_{\lambda,p}^\infty(C)$ by a constant. Now, in any unitary representation the Casimir element defines a self-adjoint operator by \cite[Theorem 4.4.4.3]{War72} which implies that $d\pi_{\lambda,p}^\infty(C)$ (or equivalently $\Delta_{a,p}$) is essentially self-adjoint on $\dot{H}^{\frac{2-a}{2},p}(\RR^n)=I_{\lambda,p}$. Further, by Corollary~\ref{cor:DiscreteSummands} the operator $d\pi_{\lambda,p}^\infty(C)$ has $\nu^2-\rho'^2+p(n-p-1)$ as an eigenvalue whenever
$$ \nu\in\left(\lambda+\frac{1}{2}+2\NN\right)\cap\left(-\frac{n-1}{2}+p,0\right)=\left\{\frac{a-1}{2}+2k:k\in\NN,k<\frac{1-a}{2}\right\}. $$
By Proposition~\ref{prop:ActionCasimir} we have $d\pi_{\lambda,p}^\infty(C)=\Delta_{a,p}+(\lambda+\rho)(\lambda-\rho+1)+p(n-p-1)$, so the operator $\Delta_{a,p}$ has the eigenvalues $(\nu+\rho')(\nu-\rho')-(\lambda+\rho)(\lambda-\rho+1)$. For $\nu=\frac{a-1}{2}+2k$ this expression equals $k(k+a-1)$. This completes the proof of (1).\\
To show (2) consider the special case $\nu=\frac{a-1}{2}$, then $\lambda+\rho-\nu-\rho'=0$, i.e. $k=0$, and therefore the operator $C_{\lambda,\nu,p}$ can be taken to be the restriction of $p$-forms on $\RR^n$ to $\RR^{n-1}$. Corollary~\ref{cor:DiscreteSummands} now implies $C_{\lambda,\nu,p}^*\circ C_{\lambda,\nu,p}=c_{\lambda,\nu,p}\cdot\id$. Hence, for every $f\in\dot{H}^{\frac{1-a}{2},p}(\RR^{n-1})=J_{\nu,p}$ the function $u=c_{\lambda,\nu,p}^{-1}C_{\lambda,\nu,p}^*f\in I_{\lambda,p}=\dot{H}^{\frac{2-a}{2},p}(\RR^n)$ satisfies $u|_{\RR^{n-1}}=C_{\lambda,\nu,p}u=f$ and by Corollary~\ref{cor:DiscreteSummands} also $\Delta_{a,p}u=0$. This establishes the existence of a solution to \eqref{eq:IntroBdyValueProblem}. Uniqueness will be shown in Section~\ref{sec:Uniqueness}.
\end{proof}

\subsection{Integral symmetry breaking operators and Poisson transforms}\label{sec:IntSBOs}

The differential symmetry breaking operators $C_{\lambda,\nu,p}$ for $\lambda+\rho-\nu-\rho'=-2k$, $k\in\NN$, arise as residues of a family $A_{\lambda,\nu,p}:I_{\lambda,p}^\infty\to J_{\nu,p}^\infty$ of symmetry breaking operators which depends meromorphically on $(\lambda,\nu)\in\CC^2$. This family of operators is for $\Re(\lambda+\nu),\Re(\nu)\gg0$ given by the convergent integral
$$ A_{\lambda,\nu,p}u(y) = \int_{\RR^n}\frac{|x_n|^{\lambda-\rho+\nu+\rho'}}{(|x'-y|^2+x_n^2)^{\nu+\rho'+1}}(i_{x-y}\varepsilon_{x-y}-\varepsilon_{x-y}i_{x-y})u(x)\,dx \qquad y\in\RR^{n-1},u\in I_{\lambda,p}^\infty, $$
and extends meromorphically in $(\lambda,\nu)\in\CC^2$ (see \cite{KS18,MO17}). In \cite{KS18} all possible poles and residues of the family $A_{\lambda,\nu,p}$ are obtained.

More important for our purpose is the adjoint of $A_{\lambda,\nu,p}$:
\begin{multline*}
 B_{\lambda,\nu,p}:=A_{-\lambda,-\nu,p}^T:J_{\nu,p}^{-\infty}\to I_{\lambda,p}^{-\infty},\\
 B_{\lambda,\nu,p}f(x) = \int_{\RR^{n-1}}\frac{|x_n|^{-\lambda-\rho-\nu+\rho'}}{(|x'-y|^2+x_n^2)^{-\nu+\rho'+1}}(i_{x-y}\varepsilon_{x-y}-\varepsilon_{x-y}i_{x-y})f(y)\,dy.
\end{multline*}

Restricting $B_{\lambda,\nu,p}$ to $J_{\nu,p}^\infty$ gives a $G'$-intertwining operator $J_{\nu,p}^\infty\to I_{\lambda,p}^{-\infty}$. Another such operator arises from Corollary~\ref{cor:DiscreteSummands} if $\lambda+\rho-\nu-\rho'=-2k$ and $\lambda\in(-\frac{n}{2}+p,0)$, $\nu\in(-\frac{n-1}{2}+p,0)$ using the embeddings $J_{\nu,p}^\infty\subseteq J_{nu,p}$ and $I_{\lambda,p}\subseteq I_{\lambda,p}^{-\infty}$:
$$ C_{\lambda,\nu,p}^*:J_{\nu,p}^\infty\hookrightarrow J_{\nu,p}\to I_{\lambda,p}\hookrightarrow I_{\lambda,p}^{-\infty}. $$
To relate $B_{\lambda,\nu,p}$ and $C_{\lambda,\nu,p}^*$ we make use of the following \textit{Multiplicity One Theorem}:

\begin{fact}[{see \cite{SZ12}}]\label{fct:MultOneThms}
Let $G={\rm O}(1,n+1)$ and $G'={\rm O}(1,n)$, then for any irreducible Casselman--Wallach representations $\pi$ of $G$ and $\tau$ of $G'$, the space of $G'$-intertwining operators $\tau\to(\pi^*)|_{G'}$ is at most one-dimensional.
\end{fact}

Here a representation $\pi$ of $G$ is called \textit{Casselman--Wallach} if it is a smooth representation on a Fr\'{e}chet space which is admissible, of moderate growth and finite under the center of the universal enveloping algebra. We note that the representations $\pi_{\lambda,p}^\infty$ of $G$ and $\tau_{\nu,p}^\infty$ of $G'$ are Casselman--Wallach.

\begin{proof}[Proof of Theorem~\ref{thm:Poisson}]
Let $\lambda=\frac{a-2}{2}$ and $\nu=\frac{a-1}{2}$, then $\lambda+\rho-\nu-\rho'=0$, so that both $B_{\lambda,\nu,p}$ and $C_{\lambda,\nu,p}^*$ define $G'$-intertwining operators $J_{\nu,p}^\infty\to I_{\lambda,p}^{-\infty}$ between the representations $\tau_{\nu,p}^\infty$ and $\pi_{\lambda,p}^{-\infty}|_{G'}=(\pi_{-\lambda,p}^\infty)^*|_{G'}$. Note that the representations $(\tau_{\nu,p}^\infty,J_{\nu,p}^\infty)$ of $G'$ and $(\pi_{-\lambda,p}^\infty,I_{-\lambda,p}^\infty)$ of $G$ are irreducible. Hence, by Fact~\ref{fct:MultOneThms} the operators $B_{\lambda,\nu,p}$ and $C_{\lambda,\nu,p}^*$ are proportional. By the proof of Theorem~\ref{thm:Dirichlet} the Poisson transform $P_{a,p}$ is a scalar multiple of $C_{\lambda,\nu,p}^*$, so it follows that $P_{a,p}=c_{a,p}B_{\lambda,\nu,p}$ for a constant $c_{a,p}$ depending only on $a$ and $p$. This shows (1) up to the computation of $c_{a,p}$ which is carried out in Section~\ref{sec:ExplicitValues}. The proof of the isometry property (2) is contained in Section~\ref{sec:Isometry}.
\end{proof}

\section{Poisson transform of constant forms}\label{sec:ExplicitValues}

In this section we compute the Poisson transform of a constant $p$-form. This is used to deduce the explicit value of the constant $c_{a,p}$, and also in Section~\ref{sec:DtoN} to compute the Dirichlet-to-Neumann map of the boundary value problem~\eqref{eq:IntroBdyValueProblem}.

Let $\omega\in\bigwedge^p\CC^{n-1}$ be a constant $p$-form on $\RR^{n-1}$. We will also view $\omega$ as a constant $p$-form on $\RR^n$ which does not contain $dx_n$. Then
\begin{align*}
 P_{a,p}\omega(x) &= c_{a,p} \int_{\RR^{n-1}} \frac{|x_n|^{1-a}}{(|x'-y|^2+x_n^2)^{\frac{n-a+2}{2}}}(i_{x-y}\varepsilon_{x-y}-\varepsilon_{x-y}i_{x-y})\omega\,dy\\
 &= c_{a,p} \int_{\RR^{n-1}} \frac{|x_n|^{1-a}}{(|y|^2+x_n^2)^{\frac{n-a+2}{2}}}(i_{(y,x_n)}\varepsilon_{(y,x_n)}-\varepsilon_{(y,x_n)}i_{(y,x_n)})\omega\,dy.
\end{align*}
We have
\begin{multline*}
 (i_{(y,x_n)}\varepsilon_{(y,x_n)}-\varepsilon_{(y,x_n)}i_{(y,x_n)})\\
 = (i_y\varepsilon_y-\varepsilon_yi_y)+x_n(i_y\varepsilon_{e_n}+i_{e_n}\varepsilon_y-\varepsilon_{e_n}i_y-\varepsilon_yi_{e_n})+x_n^2(i_{e_n}\varepsilon_{e_n}-\varepsilon_{e_n}i_{e_n}).
\end{multline*}
Since the remaining part of the integrand is an even function of $y$, the integral over $x_n(i_y\varepsilon_{e_n}+i_{e_n}\varepsilon_y-\varepsilon_{e_n}i_y-\varepsilon_yi_{e_n})$ vanishes. Further, the substition $y=x_nw$ yields
\begin{multline*}
 P_{a,p}\omega(x) = c_{a,p} \int_{\RR^{n-1}} \frac{1}{(1+|w|^2)^{\frac{n-a+2}{2}}}(i_w\varepsilon_w-\varepsilon_wi_w)\omega\,dw\\
 + c_{a,p} \int_{\RR^{n-1}} \frac{1}{(1+|w|^2)^{\frac{n-a+2}{2}}}(i_{e_n}\varepsilon_{e_n}-\varepsilon_{e_n}i_{e_n})\omega\,dw.
\end{multline*}
The second integral is easily evaluated using the Beta integral:
$$ \int_{\RR^{n-1}} \frac{1}{(1+|w|^2)^{\frac{n-a+2}{2}}}\,dw = \frac{2\pi^{\frac{n-1}{2}}}{\Gamma(\frac{n-1}{2})}\int_0^\infty(1+r^2)^{-\frac{n-a+2}{2}}r^{n-2}\,dr = \frac{\pi^{\frac{n-1}{2}}\Gamma(\frac{3-a}{2})}{\Gamma(\frac{n-a+2}{2})}. $$
For the first integral we note that
$$ i_w\varepsilon_w-\varepsilon_wi_w = \sum_{i,j=1}^{n-1}w_iw_j(i_{e_i}\varepsilon_{e_j}-\varepsilon_{e_j}i_{e_i}). $$
Integrating $w_iw_j$ with $i\neq j$ gives zero whereas for $i=j$:
\begin{multline*}
 \int_{\RR^{n-1}} \frac{w_i^2}{(1+|w|^2)^{\frac{n-a+2}{2}}}\,dz = \frac{1}{n-1}\int_{\RR^{n-1}}\frac{|w|^2}{(1+|w|^2)^{\frac{n-a+2}{2}}}\,dz\\
 = \frac{2\pi^{\frac{n-1}{2}}}{(n-1)\Gamma(\frac{n-1}{2})}\int_0^\infty(1+r^2)^{-\frac{n-a+2}{2}}r^{n}\,dr = \frac{\pi^{\frac{n-1}{2}}\Gamma(\frac{1-a}{2})}{2\,\Gamma(\frac{n-a+2}{2})}.
\end{multline*}
Putting this together gives
\begin{align*}
 P_{a,p}\omega(x) &= c_{a,p}\frac{\pi^{\frac{n-1}{2}}\Gamma(\frac{1-a}{2})}{2\,\Gamma(\frac{n-a+2}{2})}\sum_{j=1}^{n-1}(i_{e_j}\varepsilon_{e_j}-\varepsilon_{e_j}i_{e_j})\omega + c_{a,p}\frac{\pi^{\frac{n-1}{2}}\Gamma(\frac{3-a}{2})}{\Gamma(\frac{n-a+2}{2})}(i_{e_n}\varepsilon_{e_n}-\varepsilon_{e_n}i_{e_n})\omega\\
 &= c_{a,p}\frac{\pi^{\frac{n-1}{2}}\Gamma(\frac{1-a}{2})}{2\,\Gamma(\frac{n-a+2}{2})}\left(n-2p-a(i_{e_n}\varepsilon_{e_n}-\varepsilon_{e_n}i_{e_n})\right)\omega\\
 &= c_{a,p}\frac{\pi^{\frac{n-1}{2}}(n-2p-a)\Gamma(\frac{1-a}{2})}{2\,\Gamma(\frac{n-a+2}{2})}\omega.
\end{align*}
Since $P_{a,p}\omega|_{\RR^{n-1}}=\omega$ we find
$$ c_{a,p} = \frac{2\,\Gamma(\frac{n-a+2}{2})}{\pi^{\frac{n-1}{2}}(n-2p-a)\Gamma(\frac{1-a}{2})}. $$

\section{Uniqueness}\label{sec:Uniqueness}

In this section we show the uniqueness of solutions to the boundary value problem~\eqref{eq:IntroBdyValueProblem} in the homogeneous Sobolev space $\dot{H}^{\frac{2-a}{2},p}(\RR^n)$. This is done using the Euclidean Fourier transform under which the differential equation $\Delta_{a,p}u=0$ essentially corresponds to a vector-valued second order differential equation in one variable which we solve explicitly (see Theorem~\ref{thm:FTofPoisson}.

\subsection{Fourier transform of the boundary value problem}

We use the following normalization of the Euclidean Fourier transform:
$$ \calF_{\RR^n}u(\xi) = \widehat{u}(\xi) = (2\pi)^{-\frac{n}{2}}\int_{\RR^n}e^{-ix\cdot\xi}u(x)\,dx. $$
Then
$$ \widehat{x_ju}(\xi) = i\partial_{\xi_j}\widehat{u}(\xi), \qquad \widehat{\partial_{x_j}u}(\xi) = i\xi_j\widehat{u}(\xi). $$

Now assume $u\in\dot{H}^{\frac{2-a}{2},p}(\RR^n)$ is a solution of \eqref{eq:IntroBdyValueProblem}, then
\begin{align*}
 0 = \widehat{\Delta_a u}(\xi) &= \Big(\partial_{\xi_n}^2|\xi|^2-a\partial_{\xi_n}\xi_n-2\partial_{\xi_n}\big(i_{e_n}\varepsilon_{\xi'}+\varepsilon_{e_n}i_{\xi'}\big)-(n-2p)\varepsilon_{e_n}i_{e_n}\Big)\widehat{u}(\xi)\\
 &= \Big(|\xi|^2\partial_{\xi_n}^2-(a-4)\xi_n\partial_{\xi_n}-2\big(i_{e_n}\varepsilon_{\xi'}+\varepsilon_{e_n}i_{\xi'}\big)\partial_{\xi_n}-(a-2)-(n-2p)\varepsilon_{e_n}i_{e_n}\Big)\widehat{u}(\xi).
\end{align*}
and
$$ \int_\RR\widehat{u}(\xi',\xi_n)\,d\xi_n = \sqrt{2\pi}\widehat{f}(\xi'). $$
Let $z=\frac{\xi_n}{|\xi'|}$ and define
$$ v(\xi',z) := \widehat{u}(\xi',|\xi'|z), $$
then
$$ \widehat{\Delta_au}(\xi) = \calD_a v(\xi',z) \qquad \mbox{and} \qquad \int_\RR v(\xi',z)\,dz = \frac{\sqrt{2\pi}}{|\xi'|}\widehat{f}(\xi'), $$
where
$$ \calD_a = (1+z^2)\frac{d^2}{dz^2}-\Big[(a-4)z+2\big(i_{e_n}\varepsilon_{\widehat{\xi'}}+\varepsilon_{e_n}i_{\widehat{\xi'}}\big)\Big]\frac{d}{dz}-(a-2)-(n-2p)\varepsilon_{e_n}i_{e_n} $$
with $\widehat{\xi'}=\xi'/|\xi'|$. We decompose
\begin{equation}
 \widehat{f}(\xi') = \widehat{f}_{\rm I}(\xi')+\widehat{\xi}'\wedge\widehat{f}_{\rm II}(\xi')\label{eq:DecompositionF}
\end{equation}
with $i_{\xi'}\widehat{f}_R(\xi')=0$ for $R={\rm I,II}$, and similarly
$$ v(\xi',z) = \frac{\sqrt{2\pi}}{|\xi'|}\Big[v_{\rm I}(\xi',z)+\widehat{\xi}'\wedge v_{\rm II}(\xi',z)+e_n\wedge v_{\rm III}(\xi',z)+e_n\wedge\widehat{\xi}'\wedge v_{\rm IV}(\xi',z)\Big], $$
with $i_{\xi'}v_R(\xi',z)=i_{e_n}v_R(\xi',z)=0$ for $R={\rm I,II,III,IV}$. Then
\begin{align}\label{eq:NormalizationPhi}
\begin{split}
 \int_\RR v_R(\xi',z)\,dz &= \widehat{f}_R(\xi') \qquad \mbox{for }R={\rm I,II},\\
 \int_\RR v_R(\xi',z)\,dz &= 0 \qquad\qquad \mbox{for }R={\rm III,IV},
\end{split}
\end{align}
and for fixed $\xi'$ the following ODEs are satisfied:
\begin{align*}
 & \Big[(1+z^2)\frac{d^2}{dz^2}-(a-4)z\frac{d}{dz}-(a-2)\Big]v_{\rm I} = 0,\\
 & \Big[(1+z^2)\frac{d^2}{dz^2}-(a-4)z\frac{d}{dz}-(a-2)\Big]v_{\rm II}+2\frac{d}{dz}v_{\rm III} = 0,\\
 & \Big[(1+z^2)\frac{d^2}{dz^2}-(a-4)z\frac{d}{dz}-(n-2p+a-2)\Big]v_{\rm III}-2\frac{d}{dz}v_{\rm II} = 0,\\
 & \Big[(1+z^2)\frac{d^2}{dz^2}-(a-4)z\frac{d}{dz}-(n-2p+a-2)\Big]v_{\rm IV} = 0.
\end{align*}
Further, the condition $u\in\dot{H}^{\frac{2-a}{2},p}(\RR^n)$ implies that
\begin{align*}
 |u|_\lambda^2 &= \int_{\RR^n}|\xi|^{2-a}\|\widehat{u}(\xi)\|^2\,d\xi = \int_{\RR^{n-1}}|\xi'|^{3-a}\int_\RR(1+z^2)^{\frac{2-a}{2}}\|v(\xi',z)\|^2\,dz\,d\xi'
\end{align*}
is finite. Therefore, for every $R={\rm I,II,III,IV}$ we must have
\begin{equation}
 \int_\RR|v_R(\xi',z)|^2(1+z^2)^{\frac{2-a}{2}}\,dz < \infty.\label{eq:L2ConditionPhi}
\end{equation}

\subsection{Equation I} The ODE for $v_{\rm I}$ has a regular singularity at $z=+\infty$. We therefore substitute $y=z^{-1}$ and find
$$ \Big[(1+y^2)y^2\frac{d^2}{dy^2} + \big[(a-2)y+2y^3\big]\frac{d}{dy} - (a-2)\Big]v_{\rm I} = 0. $$
The corresponding indicial equation at $y=0$ is
$$ \mu(\mu-1)+(a-2)\mu-(a-2) = 0 $$
which has the two roots $\mu_1=1$ and $\mu_2=2-a$. Since $a<1$, the roots are distinct and we have $\mu_1<\frac{3-a}{2}<\mu_2$. Hence, there exist two linearly independent solutions with asymptotic behaviour $\sim y^{\mu_1}=z^{-\mu_1}$ and $\sim y^{\mu_2}=z^{-\mu_2}$ as $z\to+\infty$. If $v_{\rm I}(\xi',z)\sim z^{-\mu_1}$ as $z\to+\infty$, the integral \eqref{eq:L2ConditionPhi} diverges, whence $v_{\rm I}(\xi',z)$ has to be a scalar multiple of the solution with asymptotic behaviour $\sim z^{-\mu_2}$ as $z\to+\infty$. To find this solution we first rewrite the differential equation using $x=-z^2$ to find
$$ \Bigg[x(1-x)\frac{d^2}{dx^2}+\Big(\frac{1}{2}+\frac{a-5}{2}x\Big)\frac{d}{dx}+\frac{a-2}{4}\Bigg]\phi_{\rm I}(z) = 0, $$
which is the hypergeometric equation with $\alpha=\frac{\mu_1}{2}=\frac{1}{2}$, $\beta=\frac{\mu_2}{2}=\frac{2-a}{2}$ and $\gamma=\frac{1}{2}$. The solution with asymptotic behaviour $\sim z^{-\mu_2}=(-x)^{-\mu_2/2}$ as $z\to+\infty$ resp. $x\to-\infty$ is given by
$$ (-x)^{-\beta}{_2F_1}(\beta,1+\beta-\gamma;1+\beta-\alpha;x^{-1}) = (1-x)^{\frac{a-2}{2}} = (1+z^2)^{\frac{a-2}{2}} $$
and its asymptotic behavious as $z\to-\infty$ is also $\sim |z|^{-\mu_2}$ whence the $L^2$-condition \eqref{eq:L2ConditionPhi} is indeed satisfied. With the normalization \eqref{eq:NormalizationPhi} we find
$$ v_{\rm I}(\xi',z) = \frac{\Gamma(\frac{2-a}{2})}{\sqrt{\pi}\Gamma(\frac{1-a}{2})}(1+z^2)^{\frac{a-2}{2}}\widehat{f}_{\rm I}(\xi'), $$
where we have used the beta integral formula to compute the relevant integral.

\subsection{Equation IV} Next we consider the ODE for $v_{\rm IV}$. As above it has a regular singularity at $z=+\infty$ with corresponding indicial equation
$$ \mu(\mu-1)+(a-2)\mu-(n-2p+a-2) = 0 $$
whose two roots are
$$ \mu_1 = \tfrac{3-a}{2}-\sqrt{\left(\tfrac{3-a}{2}\right)^2+(n-2p+a-2)} \quad \mbox{and} \quad \mu_2 = \tfrac{3-a}{2}+\sqrt{\left(\tfrac{3-a}{2}\right)^2+(n-2p+a-2)}. $$
Since $2-n+2p<a<1$, the roots are also distinct and we have $\mu_1<\frac{3-a}{2}<\mu_2$. Hence, there exist two linearly independent solutions with asymptotic behaviour $\sim z^{-\mu_1}$ and $\sim z^{-\mu_2}$ as $z\to+\infty$. Again, the asymptotic behaviour $\sim z^{-\mu_1}$ as $z\to+\infty$ can be ruled out due to the $L^2$-condition \eqref{eq:L2ConditionPhi}, whence $v_{\rm IV}(\xi',z)$ has to be a scalar multiple of the solution with asymptotic behaviour $\sim z^{-\mu_2}$ as $z\to+\infty$. As above, the substitution $x=-z^2$ gives
$$ \Bigg[x(1-x)\frac{d^2}{dx^2}+\Big(\frac{1}{2}+\frac{a-5}{2}x\Big)\frac{d}{dx}+\frac{n-2p+a-2}{4}\Bigg]\phi_{\rm IV}(z) = 0, $$
which is the hypergeometric equation with $\alpha=\frac{\mu_1}{2}$, $\beta=\frac{\mu_2}{2}$ and $\gamma = \frac{1}{2}$. The solution with asymptotic behaviour $\sim z^{-\mu_2}=(-x)^{-\mu_2/2}$ as $z\to+\infty$ resp. $x\to-\infty$ is given by
\begin{multline*}
 (-x)^{-\beta}{_2F_1}(\beta,1+\beta-\gamma;1+\beta-\alpha;x^{-1}) = \frac{\Gamma(1-\gamma)\Gamma(1+\beta-\alpha)}{\Gamma(1+\beta-\gamma)\Gamma(1-\alpha)}{_2F_1}(\alpha,\beta;\gamma;x)\\
 +\frac{\Gamma(\gamma-1)\Gamma(1+\beta-\alpha)}{\Gamma(\beta)\Gamma(\gamma-\alpha)}(-x)^{1-\gamma}{_2F_1}(1+\alpha-\gamma,1+\beta-\gamma;2-\gamma;x).
\end{multline*}
Resubstituting $x=-z^2$ gives
$$ \frac{\Gamma(\frac{1}{2})\Gamma(\frac{\mu_2-\mu_1+2}{2})}{\Gamma(\frac{\mu_2+1}{2})\Gamma(\frac{2-\mu_1}{2})}{_2F_1}(\tfrac{\mu_1}{2},\tfrac{\mu_2}{2};\tfrac{1}{2};-z^2) + \frac{\Gamma(-\frac{1}{2})\Gamma(\frac{\mu_2-\mu_1+2}{2})}{\Gamma(\frac{\mu_2}{2})\Gamma(\frac{1-\mu_1}{2})}z\cdot{_2F_1}(\tfrac{\mu_1+1}{2},\tfrac{\mu_2+1}{2};\tfrac{3}{2};-z^2) $$
which has asymptotic behaviour $\sim |z|^{-\mu_1}$ as $z\to-\infty$ and therefore the $L^2$-condition \eqref{eq:L2ConditionPhi} is violated. This implies $v_{\rm IV}=0$.

\subsection{Equations II \& II} Now we treat the system of ODEs for $v_{\rm II}$ and $v_{\rm III}$:
$$ (1+z^2)\frac{d^2}{dz^2}\begin{pmatrix}v_{\rm II}\\v_{\rm III}\end{pmatrix}-\begin{pmatrix}(a-4)z&-2\\2&(a-4)z\end{pmatrix}\frac{d}{dz}\begin{pmatrix}v_{\rm II}\\v_{\rm III}\end{pmatrix}-\begin{pmatrix}a-2&0\\0&n-2p+a-2\end{pmatrix}\begin{pmatrix}v_{\rm II}\\v_{\rm III}\end{pmatrix} = 0. $$
To study the behaviour at $z=+\infty$ we again substitute $y=z^{-1}$:
\begin{multline*}
 y^2\frac{d^2}{dy^2}\begin{pmatrix}v_{\rm II}\\v_{\rm III}\end{pmatrix}+\frac{1}{1+y^2}\begin{pmatrix}(a-2)+2y^2&-2y\\2y&(a-2)+2y^2\end{pmatrix}y\frac{d}{dy}\begin{pmatrix}v_{\rm II}\\v_{\rm III}\end{pmatrix}\\
 -\frac{1}{1+y^2}\begin{pmatrix}a-2&0\\0&n-2p+a-2\end{pmatrix}\begin{pmatrix}v_{\rm II}\\v_{\rm III}\end{pmatrix} = 0.
\end{multline*}
The corresponding indicial equation at $y=0$ is
$$ \det\left(\mu(\mu-1)I+\mu\begin{pmatrix}a-2&0\\0&a-2\end{pmatrix}-\begin{pmatrix}a-2&0\\0&n-2p+a-2\end{pmatrix}\right) = 0. $$
Its four roots in increasing order are
\begin{align*}
 \mu_1 &= \tfrac{3-a}{2}-\sqrt{\left(\tfrac{a-3}{2}\right)^2+(n-2p+a-2)},\\
 \mu_2 &= 1,\\
 \mu_3 &= 2-a,\\
 \mu_4 &= \tfrac{3-a}{2}+\sqrt{\left(\tfrac{a-3}{2}\right)^2+(n-2p+a-2)}.
\end{align*}
They are distinct and therefore the system has four independent solutions with asymptotic behaviour $\sim y^{\mu_j}=z^{-\mu_j}$, $j=1,2,3,4$.
The solutions with asymptotics $\sim z^{-\mu_1}$ and $\sim z^{-\mu_2}$ can again be ruled out since they do not satisfy the $L^2$-condition \eqref{eq:L2ConditionPhi}. We now rule out the solution with asymptotics $z^{-\mu_4}$ by making the Ansatz
$$ v_{\rm II}(z) = z\cdot\varphi(-z^2) \qquad \mbox{and} \qquad v_{\rm III}(z) = \psi(-z^2). $$
Then $\varphi$ and $\psi$ solve
\begin{align}
 & x(1-x)\varphi''(x)+\left(\frac{3}{2}-\frac{7-a}{2}x\right)\varphi'(x)-\frac{3-a}{2}\varphi(x)+\psi'(x) = 0,\label{eq:HypergeomSystem1}\\
 & x(1-x)\psi''(x)+\left(\frac{1}{2}-\frac{5-a}{2}x\right)\psi'(x)+\frac{n-2p+a-2}{4}\psi(x)+x\varphi'(x)+\frac{1}{2}\varphi(x) = 0.\label{eq:HypergeomSystem2}
\end{align}
Differentiating \eqref{eq:HypergeomSystem2} once and inserting \eqref{eq:HypergeomSystem1} into the resulting equation for $\psi'$, $\psi''$ and $\psi'''$ gives
\begin{multline*}
 \left(x(1-x)\frac{d^2}{dx^2}+\left(\frac{3}{2}-\frac{7-a}{2}x\right)\frac{d}{dx}-\frac{3-a}{2}\right)\\
 \times\left(x(1-x)\frac{d^2}{dx^2}+\left(\frac{3}{2}-\frac{9-a}{2}x\right)\frac{d}{dx}+\frac{n-2p+3a-12}{4}\right)\varphi(x)=0,
\end{multline*}
and inserting \eqref{eq:HypergeomSystem1} directly into \eqref{eq:HypergeomSystem2} gives
\begin{multline*}
 \frac{n-2p+a-2}{4}\psi(x) = x^2(1-x)^2\varphi'''(x)+x(1-x)(3-(8-a)x)\varphi''(x)\\
 +\left(\frac{3}{4}+\left(2a-\frac{23}{2}\right)x+\frac{(5-a)(11-a)}{4}x^2\right)\varphi'(x)-\frac{5-a}{2}\left(\frac{1}{2}-\frac{3-a}{2}x\right)\varphi(x).
\end{multline*}
In particular, we have a solution $(\varphi,\psi)$ for every solution $\varphi$ of
$$ \left(x(1-x)\frac{d^2}{dx^2}+\left(\frac{3}{2}-\frac{9-a}{2}x\right)\frac{d}{dx}+\frac{n-2p+3a-12}{4}\right)\varphi(x) = 0 $$
which is the hypergeometric equation with
\begin{align*}
 \alpha &= \frac{1}{2}\left(\frac{7-a}{2}+\sqrt{\left(\frac{3-a}{2}\right)^2+(n-2p+a-2)}\right),\\
 \beta &= \frac{1}{2}\left(\frac{7-a}{2}-\sqrt{\left(\frac{3-a}{2}\right)^2+(n-2p+a-2)}\right),\\
 \gamma &= \tfrac{3}{2}.
\end{align*}
The solution belonging to $z^{-\mu_4}$ is
\begin{align*}
 \varphi(x) ={}& (-x)^{-\alpha}{_2F_1}(\alpha,1+\alpha-\gamma;1+\alpha-\beta;x^{-1})\\
 ={}& \frac{\Gamma(1-\gamma)\Gamma(1+\alpha-\beta)}{\Gamma(1+\alpha-\gamma)\Gamma(1-\beta)}{_2F_1}(\alpha,\beta;\gamma;x)\\
 &+\frac{\Gamma(\gamma-1)\Gamma(1+\alpha-\beta)}{\Gamma(\alpha)\Gamma(\gamma-\beta)}(-x)^{1-\gamma}{_2F_1}(1+\alpha-\gamma,1+\beta-\gamma;2-\gamma;x).
\end{align*}
Resubstituting $x=-z^2$ gives
\begin{multline*}
 \frac{\Gamma(1-\gamma)\Gamma(1+\alpha-\beta)}{\Gamma(1+\alpha-\gamma)\Gamma(1-\beta)}{_2F_1}(\alpha,\beta;\gamma;-z^2)\\
 +\frac{\Gamma(\gamma-1)\Gamma(1+\alpha-\beta)}{\Gamma(\alpha)\Gamma(\gamma-\beta)}z^{-1}{_2F_1}(1+\alpha-\gamma,1+\beta-\gamma;2-\gamma;-z^2).
\end{multline*}
Since $\alpha,\gamma-\beta>0$  the coefficient of the second factor is non-zero and therefore the solution does not extend from $(0,\infty)$ to $\RR$. Consequently, the one possible solution of equations II \& III is the one with asymptotics $\sim z^{-\mu_3}$. To find this solution we make the Ansatz
$$ v_{\rm II}(z) = \varphi(-z^2) \qquad \mbox{and} \qquad v_{\rm III}(z) = z\cdot\psi(-z^2), $$
then $\varphi$ and $\psi$ solve the system
\begin{align}
 & x(1-x)\varphi''(x)+\left(\frac{1}{2}-\frac{5-a}{2}x\right)\varphi'(x)-\frac{2-a}{4}\varphi(x)-x\psi'(x)-\frac{1}{2}\psi(x) = 0,\label{eq:HypergeomSystem3}\\
 & x(1-x)\psi''(x)+\left(\frac{3}{2}-\frac{7-a}{2}x\right)\psi'(x)+\frac{n-2p+2a-6}{4}\psi(x)-\varphi'(x) = 0.\label{eq:HypergeomSystem4}
\end{align}
 
Differentiating \eqref{eq:HypergeomSystem3} once and inserting \eqref{eq:HypergeomSystem4} into the resulting equation for $\varphi'$, $\varphi''$ and $\varphi'''$ gives
\begin{multline*}
 \left(x(1-x)\frac{d^2}{dx^2}+\left(\frac{3}{2}-\frac{7-a}{2}x\right)\frac{d}{dx}+\frac{n-2p+2a-6}{4}\right)\\
 \times\left(x(1-x)\frac{d^2}{dx^2}+\left(\frac{3}{2}-\frac{9-a}{2}x\right)\frac{d}{dx}-\frac{12-3a}{4}\right)\psi(x)=0,
\end{multline*}
and inserting \eqref{eq:HypergeomSystem4} directly into \eqref{eq:HypergeomSystem3} gives
\begin{multline*}
 \frac{2-a}{4}\varphi(x) = x^2(1-x)^2\psi'''(x)+x(1-x)(3-(8-a)x)\psi''(x)\\
 +\left(\frac{3}{4}+\frac{n-2p+8a-46}{4}x+\frac{(5-a)(11-a)-n+2p}{4}x^2\right)\psi'(x)\\
 +\left(\frac{n-2p+2a-10}{8}+\frac{(a-5)(n-2p+2a-6)}{8}x\right)\psi(x).
\end{multline*}
In particular, we have a solution $(\varphi,\psi)$ for every solution $\psi$ of
$$ \left(x(1-x)\frac{d^2}{dx^2}+\left(\frac{3}{2}-\frac{9-a}{2}x\right)\frac{d}{dx}-\frac{3(4-a)}{4}\right)\psi(x) = 0 $$
which is the hypergeometric equation with
$$ \alpha = \frac{3}{2}, \quad  \beta = \frac{4-a}{2}, \quad \gamma = \frac{3}{2}. $$
The solution belonging to $z^{-\mu_3}$ is
$$ \psi(x) = (-x)^{-\beta}{_2F_1}(\beta,1+\beta-\gamma;1+\beta-\alpha;x^{-1}) = (1-x)^{\frac{a-4}{2}} $$
and hence
$$ \varphi(x) = \frac{1}{2(2-a)}\left((n-2p-a+2)-(n-2p+a-2)x\right)(1-x)^{\frac{a-4}{2}}. $$
With the normalization \eqref{eq:NormalizationPhi} we obtain
\begin{align*}
 v_{\rm II}(\xi',z) &= \frac{\Gamma(\frac{2-a}{2})}{\sqrt{\pi}(n-2p-a)\Gamma(\frac{1-a}{2})}\big[(n-2p+a-2)z^2+(n-2p-a+2)\big](1+z^2)^{\frac{a-4}{2}}\widehat{f}_{\rm II}(\xi'),\\
 v_{\rm III}(\xi',z) &= \frac{\Gamma(\frac{2-a}{2})}{\sqrt{\pi}(n-2p-a)\Gamma(\frac{1-a}{2})}2(2-a)z(1+z^2)^{\frac{a-4}{2}}\widehat{f}_{\rm II}(\xi').
\end{align*}

\subsection{The Fourier transform of the solution}

Summarizing we find
\begin{multline}
 v(\xi',z) = \frac{\sqrt{2}\Gamma(\frac{2-a}{2})}{(n-2p-a)\Gamma(\frac{1-a}{2})}|\xi'|^{-1}\Big[(n-2p-a)(1+z^2)^{\frac{a-2}{2}}\widehat{f}_{\rm I}(\xi')\\
 +[(n-2p+a-2)z^2+(n-2p-a+2)](1+z^2)^{\frac{a-4}{2}}\widehat{\xi}'\wedge\widehat{f}_{\rm II}(\xi')\\
 +2(2-a)z(1+z^2)^{\frac{a-4}{2}}e_n\wedge\widehat{f}_{\rm II}(\xi')\Big].\label{eq:SolutionV}
\end{multline}
which implies the following:

\begin{theorem}\label{thm:FTofPoisson}
Assume $0\leq p\leq\frac{n}{2}$ and $2-n+2<a<1$. Then for $f\in\dot{H}^{\frac{1-a}{2},p}(\RR^{n-1})$ with
$$ \widehat{f}(\xi') = \widehat{f}_{\rm I}(\xi')+\widehat{\xi'}\wedge\widehat{f}_{\rm II}(\xi') $$
we have
\begin{multline*}
 \widehat{P_{a,p}f}(\xi',\xi_n) = \frac{\sqrt{2}\Gamma(\frac{2-a}{2})}{(n-2p-a)\Gamma(\frac{1-a}{2})}|\xi'|^{1-a}|\xi|^{a-4}\Big[(n-2p-a)|\xi|^2\widehat{f}_{\rm I}(\xi')\\
 +\big[(n-2p-a+2)|\xi'|^2+(n-2p+a-2)\xi_n^2\big]\widehat{\xi}'\wedge\widehat{f}_{\rm II}(\xi')\\
 +2(2-a)\xi_n|\xi'|e_n\wedge\widehat{f}_{\rm II}(\xi')\Big].
\end{multline*}
\end{theorem}

\section{Isometry}\label{sec:Isometry}

In this section we prove the isometry property in Theorem~\ref{thm:Poisson}~(2) using Theorem~\ref{thm:FTofPoisson}.

\subsection{Norm of the solution}

For a given boundary value $f$ let $u$ denote the unique solution. As before, we write $v(\xi',z)=\widehat{u}(\xi',|\xi'|z)$, then
\begin{equation}
\begin{split}
 \|u\|_\lambda^2 ={}& \int_{\RR^n}|\xi|^{-a}\left\langle\Big[\Big(\frac{n}{2}-p-\frac{a-2}{2}\Big)i_\xi\varepsilon_\xi+\Big(\frac{n}{2}-p+\frac{a-2}{2}\Big)\varepsilon_\xi i_\xi\Big]\widehat{u}(\xi),\widehat{u}(\xi)\right\rangle\,d\xi\\
 ={}& \int_{\RR^{n-1}}|\xi'|^{1-a}\int_\RR(1+z^2)^{-\frac{a}{2}}\bigg\langle\Big[\Big(\frac{n}{2}-p-\frac{a-2}{2}\Big)i_\xi\varepsilon_\xi\\
 & \hspace{5cm}+\Big(\frac{n}{2}-p+\frac{a-2}{2}\Big)\varepsilon_\xi i_\xi\Big]v(\xi',z),v(\xi',z)\bigg\rangle\,dz\,d\xi'.
\end{split}\label{eq:NormOfSolution}
\end{equation}
Using the notation of \eqref{eq:DecompositionF} and \eqref{eq:SolutionV} we have
\begin{align*}
 i_\xi\varepsilon_\xi v(\xi',z) ={}& \frac{\sqrt{2}\Gamma(\frac{2-a}{2})}{(n-2p-a)\Gamma(\frac{1-a}{2})}|\xi'|\Big[(n-2p-a)(1+z^2)^{\frac{a}{2}}\widehat{f}_{\rm I}(\xi')\\
 & \hspace{5cm}+(n-2p+a-2)z^2(1+z^2)^{\frac{a-2}{2}}\widehat{\xi}'\wedge\widehat{f}_{\rm II}(\xi')\\
 & \hspace{5cm}-(n-2p+a-2)z(1+z^2)^{\frac{a-2}{2}}e_n\wedge\widehat{f}_{\rm II}(\xi')\Big],\\
 \varepsilon_\xi i_\xi v(\xi',z) ={}& \frac{\sqrt{2}\Gamma(\frac{2-a}{2})}{(n-2p-a)\Gamma(\frac{1-a}{2})}|\xi'|\Big[(n-2p-a+2)(1+z^2)^{\frac{a-2}{2}}\widehat{\xi}'\wedge\widehat{f}_{\rm II}(\xi')\\
 & \hspace{5cm}+(n-2p-a+2)z(1+z^2)^{\frac{a-2}{2}}e_n\wedge\widehat{f}_{\rm II}(\xi')\Big].\\
\end{align*}
Together this gives
\begin{multline*}
 \Big[\Big(\frac{n}{2}-p-\frac{a-2}{2}\Big)i_\xi\varepsilon_\xi+\Big(\frac{n}{2}-p+\frac{a-2}{2}\Big)\varepsilon_\xi i_\xi\Big]v(\xi',z)\\
 = \frac{(n-2p-a+2)\Gamma(\frac{2-a}{2})}{\sqrt{2}(n-2p-a)\Gamma(\frac{1-a}{2})}|\xi'|\Big[(n-2p-a)(1+z^2)^{\frac{a}{2}}\widehat{f}_{\rm I}(\xi')\\
 + (n-2p+a-2)(1+z^2)^{\frac{a}{2}}\widehat{\xi}'\wedge\widehat{f}_{\rm II}(\xi')\Big]
\end{multline*}
and hence
\begin{align*}
 & \left\langle\Big[\Big(\tfrac{n}{2}-p-\tfrac{a-2}{2}\Big)i_\xi\varepsilon_\xi+\Big(\tfrac{n}{2}-p+\tfrac{a-2}{2}\Big)\varepsilon_\xi i_\xi\Big]v(\xi',z),v(\xi',z)\right\rangle\\
 & \hspace{.5cm} = \frac{(n-2p-a+2)\Gamma(\frac{2-a}{2})^2}{(n-2p-a)^2\Gamma(\frac{1-a}{2})^2}\Big[(n-2p-a)^2(1+z^2)^{a-1}\|\widehat{f}_{\rm I}(\xi')\|^2\\
 & \hspace{1cm} + (n-2p+a-2)\big[(n-2p+a-2)z^2+(n-2p-a+2)\big](1+z^2)^{a-2}\|\widehat{f}_{\rm II}(\xi')\|^2\Big].
\end{align*}
Inserting this into \eqref{eq:NormOfSolution} gives
\begin{multline*}
 \|u\|_\lambda^2 = \frac{\sqrt{\pi}(n-2p-a+2)\Gamma(\frac{2-a}{2})}{(n-2p-a)\Gamma(\frac{1-a}{2})}\int_{\RR^{n-1}}|\xi'|^{1-a}\Big[(n-2p-a)\|\widehat{f}_{\rm I}(\xi')\|^2\\
 + (n-2p+a-2)\|\widehat{f}_{\rm II}(\xi')\|^2\Big]\,d\xi',
\end{multline*}
where we have used the following two integral formulas which follow from the Beta integral formula:
\begin{align*}
 & \int_\RR(1+z^2)^{\frac{a-2}{2}}\,dz = \frac{\sqrt{\pi}\Gamma(\frac{1-a}{2})}{\Gamma(\frac{2-a}{2})},\\
 & \int_\RR\big[(n-2p+a-2)z^2+(n-2p-a+2)\big](1+z^2)^{\frac{a-4}{2}}\,dz = (n-2p-a)\frac{\sqrt{\pi}\Gamma(\frac{1-a}{2})}{\Gamma(\frac{2-a}{2})}.
\end{align*}

\subsection{Norm of the boundary value}

On the other hand
$$ \|f\|_\nu^2 = \int_{\RR^{n-1}}|\xi'|^{1-a}\Big\langle\Big[\Big(\tfrac{n-1}{2}-2p-\tfrac{a-1}{2}\Big)i_{\widehat{\xi}'}\varepsilon_{\widehat{\xi}'}+\Big(\tfrac{n-1}{2}-2p+\tfrac{a-1}{2}\Big)\varepsilon_{\widehat{\xi}'}i_{\widehat{\xi}'}\Big]\widehat{f}(\xi'),\widehat{f}(\xi')\Big\rangle\,d\xi'. $$
Here
$$ i_{\widehat{\xi}'}\varepsilon_{\widehat{\xi}'}\widehat{f}(\xi') = \widehat{f}_{\rm I}(\xi') \qquad \mbox{and} \qquad \varepsilon_{\widehat{\xi}'}i_{\widehat{\xi}'}\widehat{f}(\xi') = \widehat{\xi}'\wedge\widehat{f}_{\rm II}(\xi'), $$
so that
$$ \|f\|_\nu^2 = \frac{1}{2} \int_{\RR^{n-1}} |\xi'|^{1-a} \Big[(n-2p-a)\|\widehat{f}_{\rm I}(\xi')\|^2+(n-2p+a-2)\|\widehat{f}_{\rm II}(\xi)\|^2\Big]\,d\xi' $$
and hence
$$ \|u\|_\lambda^2 = \frac{2\sqrt{\pi}(n-2p-a+2)\Gamma(\frac{2-a}{2})}{(n-2p-a)\Gamma(\frac{1-a}{2})}\|f\|_\nu^2. $$

\section{Dirichlet-to-Neumann map}\label{sec:DtoN}

In this section we prove Theorem~\ref{thm:DtoN}. For $s\in(0,1)$ let $a=1-2s\in(-1,1)$. Given $f\in\dot{H}^{\frac{1-a}{2},p}(\RR^{n-1})$ let $u=P_{a,p}f\in\dot{H}^{\frac{2-a}{2},p}(\RR^n)$ be the unique solution to \eqref{eq:IntroBdyValueProblem}. We use the difference quotient to compute the derivative:
\begin{align*}
  \lim_{x_n\searrow0}x_n^a\partial_{x_n}u(x',x_n) &= \lim_{x_n\searrow0}x_n^{a-1}(u(x',x_n)-u(x',0))\\
  &= \lim_{x_n\searrow0}x_n^{a-1}(P_{a,p}f(x',x_n)-f(x')).
\end{align*}
Viewing $f(x')$ as a constant form on $\RR^{n-1}$ and using Section~\ref{sec:ExplicitValues} we can write
\begin{align*}
  &= c_{a,p}\cdot\lim_{x_n\searrow0} \int_{\RR^{n-1}} \frac{1}{(|x'-y|^2+x_n^2)^{\frac{n-a+2}{2}}}(i_{x-y}\varepsilon_{x-y}-\varepsilon_{x-y}i_{x-y})(f(y)-f(x'))\,dy\\
  &= c_{a,p}\cdot\PV\int_{\RR^{n-1}} \frac{1}{|x'-y|^{n-a+2}}(i_{x'-y}\varepsilon_{x'-y}-\varepsilon_{x'-y}i_{x'-y})(f(y)-f(x'))\,dy\\
  &= c_{a,p}\Gamma(-s)L_{s,p}f(x').
\end{align*}

\appendix

\section{The Casimir computation}\label{sec:Casimir}

In this appendix we compute the action of the Casimir element $C$ of $G'$ in an arbitrary principal series representation $\pi_{\lambda,\xi}^\infty$ of $G$.

An explicit basis of the Lie algebra $\frakg=\frako(1,n+1)$ is given by the generator $H$ of $\fraka$ and the elements
\begin{align*}
 M_{jk} &:= E_{j+2,k+2}-E_{k+2,j+2}, && 1\leq j<k\leq n,\\
 X_j &:= E_{j+2,1}-E_{j+2,2}+E_{1,j+2}+E_{2,j+2}, && 1\leq j\leq n,\\
 \overline{X}_j &:= E_{j+2,1}+E_{j+2,2}+E_{1,j+2}-E_{2,j+2}, && 1\leq j\leq n.
\end{align*}
Here $M_{jk}$ span $\frakm$, $X_j$ span $\frakn$ and $\overline{X}_j$ span $\overline{\frakn}$.

The action of the generators $MA\overline{N}$ and $w_0$ of $G$ on
$I_{\lambda,\xi}^\infty$ is then given by
\begin{align*}
 \pi_{\lambda,\xi}^\infty(\overline{n}_{x'})f(x) &= f(x-x'), && \overline{n}_x\in\overline{N},\\
 \pi_{\lambda,\xi}^\infty(\diag(\lambda,\lambda,m))f(x) &= \xi(m)f(\lambda m^{-1}x), && \lambda\in {\rm O}(1),m\in {\rm O}(n),\\
 \pi_{\lambda,\xi}^\infty(e^{sH})f(x) &= e^{(\lambda+\rho)s}f(e^sx), && e^{sH}\in A,\\
 \pi_{\lambda,\xi}^\infty(w_0)f(x) &= \xi\Big(\1_n-2\tfrac{xx^\top}{|x|^2}\Big)|x|^{-2(\lambda+\rho)}f\left(-\frac{x}{|x|^2}\right).
\end{align*}
By differentiation we obtain the action of $\frakm$, $\fraka$ and $\overline{\frakn}$:
$$ d\pi_{\lambda,\xi}^\infty(M_{jk}) = x_j\frac{\partial}{\partial x_k}-x_k\frac{\partial}{\partial x_j}+d\xi(M_{jk}), \qquad d\pi_{\lambda,\xi}^\infty(H) = E+\lambda+\rho, \qquad d\pi_{\lambda,\xi}^\infty(\overline{X}_j) = -\frac{\partial}{\partial x_j}, $$
where $E=\sum_{k=1}^nx_k\frac{\partial}{\partial x_k}$ denotes the Euler operator on $\RR^n$. Thanks to the relation $\Ad(w_0)\overline{X}_j=-X_j$ we have
$$ d\pi_{\lambda,\xi}^\infty(X_j) = -\pi_{\lambda,\xi}^\infty(w_0)d\pi_{\lambda,\xi}^\infty(\overline{X}_j)\pi_{\lambda,\xi}^\infty(w_0), $$
which is easily shown to be equal to
$$ d\pi_{\lambda,\xi}^\infty(X_j) = -|x|^2\frac{\partial}{\partial x_j}+2x_j(E+\lambda+\rho)-2\sum_{i=1}^nx_i\,d\xi(M_{ij}). $$
Note that $M_{jj}=0$ and $M_{ij}=-M_{ji}$.

The Casimir element $C$ can be expressed using the above constructed basis of $\frakg$:
$$ C = H^2-(n-1)H-\sum_{1\leq j<k\leq n-1}M_{jk}^2+\sum_{j=1}^{n-1}X_j\overline{X}_j. $$
An elementary calculation using the previously derived formulas for the differential representation shows that
\begin{multline}
 d\pi_{\lambda,\xi}^\infty(C) = x_n^2\Delta+2(\lambda+1)x_n\frac{\partial}{\partial x_n}
-2x_n\sum_{j=1}^{n-1}d\xi(M_{jn})\frac{\partial}{\partial x_j}\\
 -\sum_{1\leq j<k\leq n-1}d\xi(M_{jk})^2+(\lambda+\rho)(\lambda-\rho+1),
\label{eq:FormulaCasimirReal}
\end{multline}
where $\Delta=\sum_{k=1}^n\frac{\partial^2}{\partial x_k^2}$
denotes the Euclidean Laplacian on $\RR^n$.


\section{Comparison with the Laplace--Beltrami operator on the hyperbolic upper half plane}\label{sec:Laplacian}

In this appendix we compute the Laplace-Beltrami operator on $p$-forms on the hyperbolic upper half plane
$$ \HH^n=\{x\in\RR^n:x_n>0\} $$
and compare it with the operator $\Delta_{a,p}$. The upper half plane $\HH^n$ is equipped with the hyperbolic metric $g=x_n^{-2}(dx_1^2 +\cdots +dx_n^2)$, so that $g_{ij}=x_n^{-2}\delta_{ij}$. The Riemannian volume form is $dv=x_n^{-n} dx$ since $\det(g_{ij})=x_n^{-2n}$. Further, the corresponding inner product $\llangle\cdot,\cdot\rrangle_x$ at $x\in\HH^n$ is given by
$$ \llangle\alpha(x),\beta(x)\rrangle_x = x_n^{2p}\langle\alpha(x),\beta(x)\rangle \qquad \forall\,\alpha,\beta\in\Omega^p(\HH^n), $$
where $\langle\cdot,\cdot\rangle$ is the standard inner product induced from the Euclidean metric. We write
$$ \llangle\alpha,\beta\rrangle = \int_{\HH^n}\llangle\alpha(x),\beta(x)\rrangle_x\,dv = \int_{\HH^n}\langle\alpha(x),\beta(x)\rangle x_n^{2p-n}\,dx $$
for the corresponding inner product on $p$-forms $\alpha,\beta$.

We first compute the codifferential $d^\ast$ on $\Omega^p(\HH^n)$. For this we write $\alpha=\sum_I \alpha_I dx_I\in\Omega^{p-1}(\HH^n)$ and $\beta=\sum_J\beta_J dx_J\in\Omega^p(\HH^n)$, then
\begin{align*}
 \llangle\alpha,d^*\beta\rrangle &= \llangle d\alpha,\beta\rrangle\\
 &= \int_{\HH^n} \langle d\alpha(x), \beta(x)\rangle x_n^ {2p-n}\,dx\\
 &= \sum_{I,J}\sum_{j=1}^n \int_{\HH^n} \frac{\partial\alpha_I}{\partial x_j}(x)\beta_J(x)\langle \varepsilon_{e_j}dx_I,dx_J\rangle x_n^{2p-n}\,dx\\
 &= -\sum_{I,J}\sum_{j=1}^n \int_{\HH^n} \alpha_I(x)\frac{\partial}{\partial x_j}\big[\beta_J(x)x_n^{2p-n}\big]\langle dx_I,i_{e_j}dx_J\rangle\,dx,
\end{align*}
where we have integrated by parts in the last step. Clearly,
$$ \frac{\partial}{\partial x_j}\big[\beta_J(x)x_n^{2p-n}\big] = \frac{\partial\beta_J}{\partial x_j}(x)x_n^{2p-n}+(2p-n)\delta_{jn}\beta_J(x)x_n^{2p-n-1}, $$
so that
\begin{multline*}
 \llangle\alpha,d^*\beta\rrangle = -\sum_{I,J}\sum_{j=1}^n\int_{\HH^n}\Big\langle\alpha_I(x)dx_I,x_n^2\frac{\partial\beta_J}{\partial x_j}(x)i_{e_j}dx_J\\
 +(2p-n)\delta_{jn}x_n\beta_J(x)i_{e_n}dx_J\Big\rangle x_n^{2(p-1)-n}\,dx.
\end{multline*}
It follows that
$$ d^* = x_n^2\delta-(2p-n)x_ni_{e_n}, $$
where $\delta=-\sum_ji_{e_j}\frac{\partial}{\partial x_j}$ is the Euclidean codifferential. Therefore, the hyperbolic Laplace--Beltrami operator $\square_p=-(d^*d+dd^*)$ on $p$-forms takes the following form:
$$ \square_p = -x_n^2(\delta d + d\delta) + (2(p+1)-n) x_ni_{e_n} d + (2p-n)(\varepsilon_{e_n}i_{e_n} + x_n di_{e_n}) - 2x_n \varepsilon_{e_n}\delta. $$
The first term is $-x_n^2(\delta d + d\delta)=x_n^2\Delta$, where $\Delta=\sum_j\frac{\partial^2}{\partial x_j^2}$ is the Euclidean Laplacian. In the second, third and fourth term we can write
\begin{align*}
 i_{e_n}d = i_{e_n}d' + i_{e_n}\varepsilon_{e_n}\frac{\partial}{\partial x_n}, \qquad di_{e_n} = \frac{\partial}{\partial x_n} - i_{e_n} d \qquad \varepsilon_{e_n}\delta = \varepsilon_{e_n}\delta' - \varepsilon_{e_n}i_{e_n}\frac{\partial}{\partial x_n},
\end{align*}
with $d'$ and $\delta'$ as in \eqref{eq:DefDandDelta}. Using $\varepsilon_{e_n}i_{e_n}+i_{e_n}\varepsilon_{e_n}=\id$ this gives
$$ \square_p = x_n^2 \Delta + (2(p+1)-n) x_n\frac{\partial}{\partial x_n} + 2x_n(i_{e_n}d'-\varepsilon_{e_n}\delta') - (n-2p)\varepsilon_{e_n}i_{e_n} = \Delta_{2(p+1)-n,p}. $$

We note that for any $a\in\RR$ we have
$$ \square_p\Big[x_n^{\frac{a+n-2p-2}{2}}u(x)\Big] = x_n^{\frac{a+n-2p-2}{2}}\Big[\Delta_{a,p}+\frac{1}{4}(a+n-2p-2)(a-n+2p)\Big]u(x), $$
so that the map $u(x)\mapsto x_n^{\frac{a+n-2p-2}{2}}u(x)$ transforms the Laplace--Beltrami operator $\square_p$ up to a constant into the operator $\Delta_{a,p}$. For the case $p=0$ this was already observed in \cite[end of Section 1.1]{MOZ16a}.

\providecommand{\bysame}{\leavevmode\hbox to3em{\hrulefill}\thinspace}

\providecommand{\href}[2]{#2}


\begin{thebibliography}{10}

\bibitem{BG05}
T.~Branson and A.~R. Gover, \emph{Conformally invariant operators, differential
  forms, cohomology and a generalisation of {$Q$}-curvature}, Comm. Partial
  Differential Equations \textbf{30} (2005), no.~10-12, 1611--1669.

\bibitem{CS07}
L.~Caffarelli and L.~Silvestre, \emph{An extension problem related to the
  fractional {L}aplacian}, Comm. Partial Differential Equations \textbf{32}
  (2007), no.~7-9, 1245--1260.

\bibitem{CG11}
S.-Y.~A. Chang and M.~del~Mar Gonz\'{a}lez, \emph{Fractional {L}aplacian in
  conformal geometry}, Adv. Math. \textbf{226} (2011), no.~2, 1410--1432.

\bibitem{FJS16}
M.~Fischmann, A.~Juhl, and P.~Somberg, \emph{Conformal symmetry breaking
  differential operators on differential forms},  (2016), to appear in Mem.
  Amer. Math. Soc., available at
  \href{http://arxiv.org/abs/1605.04517}{arXiv:1605.04517}.

\bibitem{FO17}
M.~Fischmann and B.~{\O}rsted, \emph{A family of {R}iesz distributions for
  differential forms on {E}uclidian space},  (2017), preprint, available at
  \href{http://arxiv.org/abs/1702.00930}{arXiv:1702.00930}.

\bibitem{Kob15}
T.~Kobayashi, \emph{A program for branching problems in the representation
  theory of real reductive groups}, Representations of reductive groups, Progr.
  Math., vol. 312, pp.~277--322.

\bibitem{KKP16}
T.~Kobayashi, T.~Kubo, and M.~Pevzner, \emph{Conformal symmetry breaking
  operators for differential forms on spheres}, Lecture Notes in Mathematics,
  vol. 2170, Springer, Singapore, 2016.

\bibitem{KS18}
T.~Kobayashi and B.~Speh, \emph{Symmetry breaking for representations of rank
  one orthogonal groups {II}}, Lecture Notes in Mathematics, vol. 2234,
  Springer, Singapore, 2018.

\bibitem{MO17}
J.~M\"{o}llers and B.~{\O}rsted, \emph{{K}napp--{S}tein type intertwining
  operators for symmetric pairs {II}. -- {T}he translation principle and
  intertwining operators for spinors},  (2017), preprint, available at
  \href{http://arxiv.org/abs/1702.02326}{arXiv:1702.02326}.

\bibitem{MOZ16a}
J.~M\"{o}llers, B.~{\O}rsted, and G.~Zhang, \emph{On boundary value problems
  for some conformally invariant differential operators}, Comm. Partial
  Differential Equations \textbf{41} (2016), no.~4, 609--643.

\bibitem{SV11}
B.~Speh and T.~N. Venkataramana, \emph{Discrete components of some
  complementary series}, Forum Math. \textbf{23} (2011), no.~6, 1159--1187.

\bibitem{SZ12}
B.~Sun and C.-B. Zhu, \emph{Multiplicity one theorems: the {A}rchimedean case},
  Ann. of Math. (2) \textbf{175} (2012), no.~1, 23--44.

\bibitem{War72}
G.~Warner, \emph{Harmonic analysis on semi-simple {L}ie groups. {I}},
  Springer-Verlag, New York, 1972, Die Grundlehren der mathematischen
  Wissenschaften, Band 188.

\end{thebibliography}
\end{document}